\documentclass{article}
 
\usepackage{breqn}
 
\usepackage{hyperref}
\usepackage[table,dvipsnames]{xcolor}
 \hypersetup{
    colorlinks=true,
    linkcolor=Maroon,     
    urlcolor=blue,
    citecolor= blue
    } 
   
\usepackage{enumitem,pifont,float,framed,fullpage}
\makeatletter
\def\namedlabel#1#2{\begingroup
    #2%
    \def\@currentlabel{#2}%
    \phantomsection\label{#1}\endgroup
}
\makeatother

\usepackage[round]{natbib}

\usepackage{graphicx}
\usepackage{amsmath}
\usepackage{amssymb}
\usepackage{amsthm}
\usepackage{thmtools}
\usepackage{thm-restate}
\usepackage{soul}

\usepackage{multirow}
\usepackage{booktabs} 
\usepackage{algorithm}
\usepackage{algorithmic}
\DeclareMathOperator*{\argmin}{argmin}

\DeclareMathOperator*{\bbR}{\mathbb{R}}
\DeclareMathOperator*{\sZ}{\mathcal{Z}}
\DeclareMathOperator*{\grad}{\operatorname{grad}}

\newcommand{\eg}{\textit{e.g.}\ }

\newcommand{\diag}{\operatorname*{diag}}
\newcommand{\sign}{\operatorname*{sign}}
\newcommand{\R}{\mathbb{R}}

\newtheorem{theorem}{Theorem}
\newtheorem{corollary}{Corollary}
\newtheorem{lemma}{Lemma}[section]

\newtheorem{remark}{Remark}
\newtheorem{definition}{Definition}
\newtheorem*{theorem*}{Theorem}

\begin{document}

\title{From the simplex to the sphere: Faster constrained optimization using the Hadamard parametrization}
 
\author{Qiuwei Li\thanks{Alibaba DAMO Academy, Bellevue, WA}      
\and
Daniel McKenzie\thanks{Colorado School of Mines, Golden, CO}
\and
Wotao Yin$^{\ast}$ 
}

\maketitle

\begin{abstract}
The standard simplex in $\mathbb{R}^n$, also known as the probability simplex, is the set of nonnegative vectors whose entries sum up to 1. They frequently appear as constraints in optimization problems that arise in machine learning, statistics, data science, operations research, and beyond. 
We convert the standard simplex to the unit sphere and thus transform the corresponding constrained optimization problem into an optimization problem on a simple, smooth manifold.
We show that KKT points and strict-saddle points of the minimization problem on the standard simplex all correspond to those of the transformed problem, and vice versa. So, solving one problem is equivalent to solving the other problem. Then, we propose several simple, efficient, and projection-free algorithms using the manifold structure. The equivalence and the proposed algorithm can be extended to optimization problems with unit simplex, weighted probability simplex, or $\ell_1$-norm sphere constraints. Numerical experiments between the new algorithms and existing ones show the advantages of the new approach. Open source code is available at \href{https://github.com/DanielMckenzie/HadRGD}{here}.
\end{abstract}

\section{Introduction}
\label{sec:Intro}
Let  $[x]_i$ denote the $i$th entry of vector $x\in\mathbb{R}^n$. Consider the constrained optimization problem:
\begin{equation}
\operatorname*{minimize}_{x \in \Delta_n} f(x)
    \tag{$\mathrm{P_{1a}}$}
    \label{C_1}
\end{equation}
where   
\begin{equation}
\Delta_n= \left\{x \in \mathbb{R}^{n}: \ \sum_{i=1}^n [x]_i = 1 \text{ and } x\ge 0 \right\} 
\end{equation}
is the {\em standard simplex}
and $f$ is twice continuously differentiable (henceforth: {\em smooth}).

A natural approach to this problem is {\em Projected Gradient Descent (PGD)}: 
\begin{equation}
    x_{k+1} = \mathrm{P}_{\Delta_n}\left(x_k - \eta_k\nabla f(x_k)\right)
\end{equation}
where $\eta_k$ is step size and $\mathrm{P}_{\Delta_n}(y)$ is the projection onto $\Delta_n$. One can also apply an appropriate interior-point method based on the type of $f$ or solve an entropy-regularization approximation of this problem.



We propose to use the {\em Hadamard parametrization}: $x = z\odot z$, where $\odot$ represents the elementwise product, that is, $[x]_i = [z]_i^2$ for $i=1,\dots,n$. 
By
\begin{equation*}
 \sum_i [x]_i = 1 \text{ and } x\geq 0,~\iff~  \|z\|_2^2=1
\end{equation*}
we transform the original set $\Delta_n$ to a unit sphere constraint $z\in\mathcal{S}_{n-1}$, 
where $\mathcal{S}_{n-1} := \{z \in \mathbb{R}^{n}: \ \|z\|_2 = 1\}$ is the unit sphere; see Fig. \ref{fig:simplex2ball}.
\begin{figure}[h]
    \centering
    \includegraphics[width=0.6\textwidth]{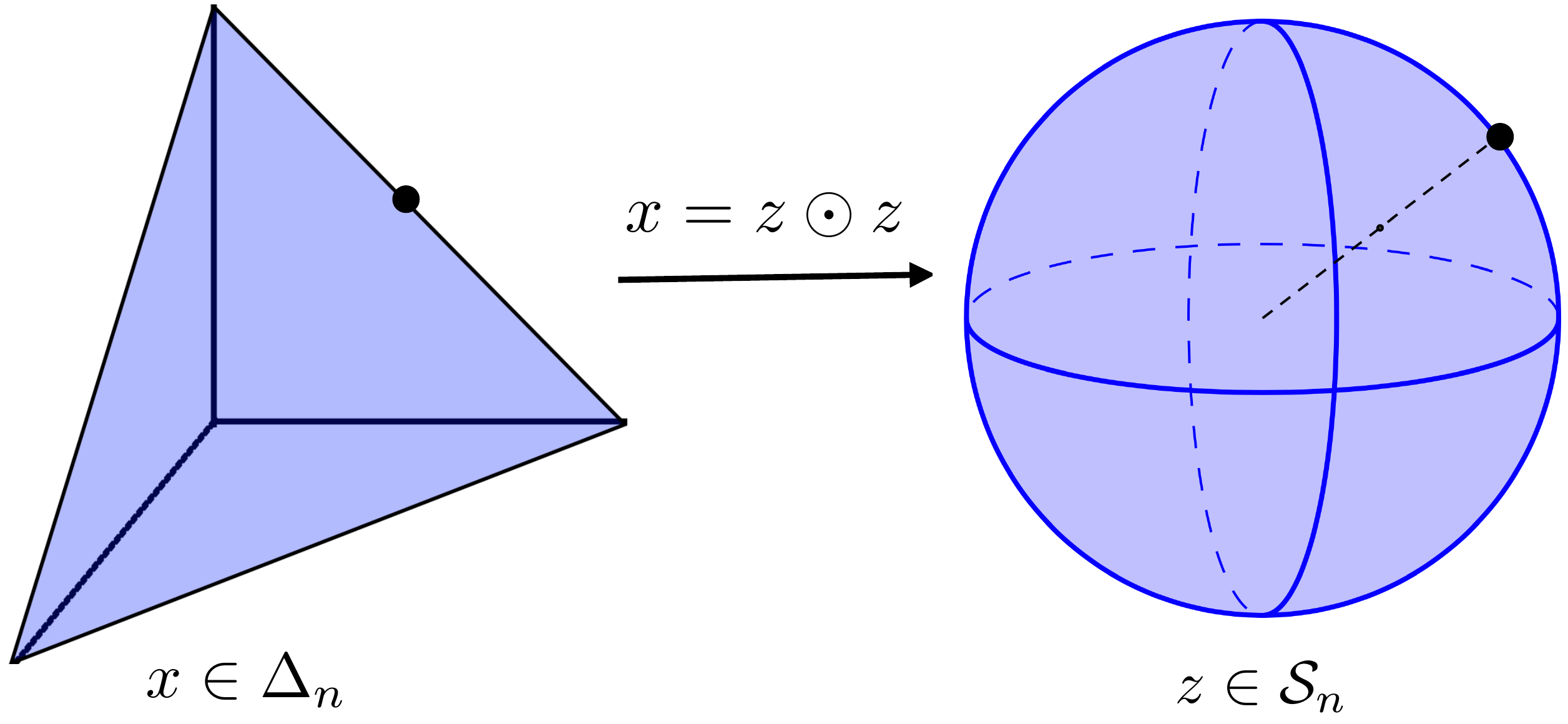}
    \caption{$x\in\Delta_n$ in the original space is transformed as $z\in\mathcal{S}_{n-1}$ in the Hadamard-parametrization space.}
    \label{fig:simplex2ball}
\end{figure}
Meanwhile, we transform \eqref{C_1} into the equivalent problem
\begin{equation}
    \operatorname*{minimize}_{z \in \mathcal{S}_{n-1}} \left\{g(z) := f(z\odot z)\right\}
    \tag{$\mathrm{P_{1b}}$}.
    \label{NC_1}
\end{equation}
Importantly $\mathcal{S}_{n-1}$ is a simple Riemannian manifold, unlike $\Delta_n$. So, techniques from Riemannian optimization 
may be applied to \eqref{NC_1}. Although it is clear the global minimizers of $f$ and $g$ must coincide,\footnote{more accurately, if $z^\star$ is a global minimizer of $g$ then $x^\star = z^\star\odot z^\star$ is a global minimizer of $f$} it is far from obvious what the correspondence between the various solutions of \eqref{C_1} and \eqref{NC_1} should be. We settle this question and show:

\begin{theorem*}[informally stated]
In general, $z^\star$ is a global minimizer, a second-order KKT point, or a strict-saddle point of \eqref{NC_1} if and only if $x^\star = z^\star\odot z^\star$ is a corresponding point of \eqref{C_1}. 

When $f$ is convex, $z^\star$ is a second-order KKT point of \eqref{NC_1} if and only if  $x^\star = z^\star\odot z^\star$is a global minimizer of \eqref{C_1}.
\end{theorem*}

Even though the conversion causes \eqref{NC_1} to be generally non-convex, as long as \eqref{C_1} is convex, \eqref{NC_1} enjoys the computational benefits of convexity.

Leveraging the theorem, we show a perturbed version of Riemannian Gradient Descent \citep{criscitiello2019efficiently,sun2019escaping}, applied to problem \eqref{NC_1}, converges to a second-order KKT point of Problem \eqref{C_1}. In many applications \cite{li2020global,li2019geometry,li2019non,zhu2021global,zhu2018global,sun2018geometric,sun2016complete}, although $f$ is non-convex it turns out that all local minima are in fact global minima. When these global minima are in addition isolated, we show our approach finds an $\varepsilon$-approximate global minimizer in $\mathcal{O}(\log^4 n/\varepsilon)$ iterations. We also extend the results above to different but related constraint sets such as the unit simplex, the weighted simplex, as well as the $\ell_1$ ball\footnote{The details, including their definitions, are given in Section \ref{sec:extensions}}.

Modern gradient methods are not used without adaptation techniques. Hence, we equip our algorithm with line search and compare it with PGD. 
We test an Armijo-Wolfe line search rule, as well as a Barzilei-Borwein (BB) step-size rule with non-monotone line search~\cite{wen2013feasible}. We observe that our proposed algorithm with a BB step-size rule is significantly faster than PGD with line search for a convex test function (underdetermined least squares). 

\paragraph{Notation}
Our original objective function will always be $f$ and by $g$ we always mean the converted version of $f$: $g(z) = f(z\odot z)$. While $\nabla$ is used as the (Euclidean) gradient operator, $\mathrm{grad}$ is reserved for the Riemannian gradient operator. We use $\nabla^2$ and $\mathrm{Hess}$ to denote the (Euclidean) Hessian and Riemannian Hessian operator, respectively. Finally, by $\mathrm{int}(\Delta_n)$ we mean the interior of $\Delta_n$.
 
\paragraph{Hadamard Calculus}
\label{sec:HadCalc}
We recall a few properties of the Hadamard product. We defer all proofs to the appendix. 
\begin{description}
    \item[\namedlabel{H1}{($\mathrm{H1}$)}]  $1_n\odot z = z$ where $1_n$ is the all-one vector in $\R^n$
    \item[\namedlabel{H2}{($\mathrm{H2}$)}] If $d\odot z\odot z = 0$ then $d\odot z = 0$
    \item[\namedlabel{H3}{($\mathrm{H3}$)}] 
        $ \mathrm{diag}(z)d = z\odot d $ and
        $ d^{\top}\mathrm{diag}(z)d = \langle d, z\odot d \rangle= \langle z,d\odot d \rangle $
    \item[\namedlabel{H4}{($\mathrm{H4}$)}] $\|d\odot z\|_2 \leq \|d\|_2\|z\|_{\infty}$

    \item[\namedlabel{H5}{($\mathrm{H5}$)}] $\nabla g(z) \!=\! 2 \nabla f(x) \odot z$   and 
    $\nabla^2 g(z) \!=\! 2\mathrm{diag}(\nabla f(x))+4 \mathrm{diag}(z) \nabla^2 f(x) \mathrm{diag}(z)$
\end{description}

\section{Prior and Related Work}
\label{sec:PriorWork}
\paragraph{Hadamard parametrization} There has been a flurry of recent papers examining the Hadamard parametrization. In particular, we highlight the papers \citep{vavskevivcius2019implicit,zhao2019implicit}, which study the basis--pursuit problem:
\begin{equation}
    \operatorname*{minimize}_x \ \|x\|_1 \quad\mathrm{s.t.}\quad Ax = b.
    \label{eq:Lasso}
\end{equation}
Informally, their approach is to set $x = z\odot z$ and then apply carefully initialized gradient descent to the {\em unconstrained} nonconvex problem:
\begin{equation}
    \operatorname*{minimize}_{z}\ \|Az\odot z - b\|_{2}^{2}
    \label{eq:Sparse_GD}
\end{equation}
Using results on the implicit bias of gradient descent \citep{you2020robust,ali2019continuous,bauer2007regularization,buhlmann2003boosting}, they argue gradient descent applied to \eqref{eq:Sparse_GD} will find a $z^{\star} \in \{z: \ Az\odot z = b\}$ of minimal $\ell_2$ norm. By $\|z\|_2^{2} = \|z\odot z\|_1$, they show $x^{\star} := z^{\star}\odot z^{\star}$ is the solution to \eqref{eq:Lasso}. 

Although inspired by these works, our approach is a distinct application. Instead of exploiting implicit regularization, we convert a non-smooth constraint set into a smooth one. 

Before uploading the second arXiv report of this manuscript, we learnt of the work \cite{levin2022effect}, which considers a generalization of the correspondence between Problems \eqref{C_1} and \eqref{NC_1} considered here. Specializing their results to our case, they show that if $z^\star$ is a second-order stationary point of \eqref{NC_1} then $x^\star := z^\star\odot z^\star$ is a first-order stationary point of \eqref{C_1}. Our main result Theorem~\ref{thm:H1} generalizes this.

\paragraph{Simplex Minimization} Alternative approaches to solving \eqref{C_1} include the Frank-Wolfe algorithm \citep{frank1956algorithm, jaggi2013revisiting}, also known as the conditional gradient algorithm, and mirror descent \citep{ben2001ordered, beck2003mirror}, known in this context as the entropic mirror descent algorithm or the exponentiated gradient algorithm. When $f$ is strongly convex certain variants of Frank-Wolfe achieve a linear convergence rate \citep{lacoste2015global, pedregosa2020linearly} although for merely convex $f$ the convergence rate is sublinear. Interior point methods ({\em e.g.} \citep{koh2007interior}) enjoy a fast convergence rate but, as the cost of each iteration is typically quadratic in $n$, are intractable for high-dimensional problems.\\

\paragraph{Avoiding Strict Saddle Points} Standard theory (\eg \cite[Chpt. 3]{bertsekas1997nonlinear}) shows that when $f$ is convex Projected Gradient Descent (PGD) finds an $\epsilon$-optimal solution to Problem \eqref{C_1} in $\mathcal{O}(1/\epsilon)$ iterations, but our focus is on the case where $f$ is non-convex, where finding $\epsilon$-optimal solutions is known to be NP-hard. For unconstrained, non-convex problems it is by now well-known that perturbed versions of simple first-order methods ({\em e.g.} gradient descent) can avoid strict saddle points---first order stationary points that are not second order stationary points---thus converging to second-order stationary points with high probability \citep{sun2015complete, ge2015escaping,jin2017escape,panageas2019first,carmon2020first,royer2020newton}.  For constrained optimization, the situation is less clear. Indeed consider the general constrained optimization problem with convex constraints:
\begin{equation}
   \operatorname*{minimize}_{x \in \mathcal{C}} f(x) 
\end{equation}
\citep{nouiehed2018convergence} shows that PGD applied to a constrained optimization problem {\em with even a single linear inequality constraint} may converge to a strict saddle point. On the other hand, when $\mathcal{C} = \{x: h_i(x) = 0 \text{ for } i=1,\ldots,m\}$ \citep{ge2015escaping} shows a perturbed version of Projected Gradient Descent converges to a second-order Karush-Kuhn-Tucker point (KKT point for short, see Section ~\ref{sec:LandscapeAnalysis}) while \citep{lu2019perturbed} shows this when $\mathcal{C}$ is a box. When $\mathcal{C}$ is a linear subspace, \cite{avdiukhin2019escaping} shows the noisy sticky projected gradient descent converges to an approximate second-order KKT point if the projector onto the linear subspace is {\em computable}. When $\mathcal{C}$ is a smooth manifold, \citep{criscitiello2019efficiently,sun2019escaping}  show a perturbed version of Riemannian gradient descent avoids (Riemannian) strict saddles. Finally, for general $\mathcal{C}$, \cite{mokhtari2018escaping} provides a hybrid first-second order method guaranteed to converge to a second-order KKT point, although the rate of convergence depends on the complexity of solving non-convex quadratic programs over $\mathcal{C}$. We provide a simple, first-order algorithm with $\mathcal{O}(n)$ per-iteration complexity that converges to a second-order KKT point of Problem \eqref{C_1} with overwhelming probability.

\paragraph{Simplex Projection}  Computing $P_{\Delta_n}$ can be a substantial part of the per-iteration computational cost of PGD \citep{perez2020filtered,condat2016fast,blondel2014large,wang2013projection,chen2011projection}. To the best of our knowledge the fastest algorithm for computing $P_{\Delta_n}$ is that of \citep{perez2020filtered}, which enjoys $\mathcal{O}(n)$ complexity, but involves a number of non-standard computational steps. We also note that the algorithm of \citep{condat2016fast} has $\mathcal{O}(n)$ expected complexity. To the best of the authors' knowledge, all competitive algorithms for computing $P_{\Delta}$ involving a sorting or median selection sub-routine, which are tricky to parallelize. Thus when one needs to accelerate algorithms by parallel and distributed computation, it may be preferable to solve Problem~\eqref{NC_1} instead of Problem~\eqref{C_1}. 

\paragraph{Applications} Typical applications of our results include, but are not limited to, estimation of mixture proportions \citep{keshava2003survey}, probability density estimation \citep{bunea2010spades},  aggregation learning \citep{nemirovski2000ecole}, 
training of support vector machines \citep{clarkson2010coresets},
portfolio optimization \citep{bomze2002regularity}, population dynamics \citep{zeeman1980population}, estimation of crossing numbers in certain classes of graphs \citep{de2006improved}, Bayesian compressive sensing over the simplex \citep{limmer2018neural}, and archetypal analysis in machine learning \citep{bauckhage2015archetypal}.

\section{Riemannian Optimization}
\label{sec:Riem_Opt}
We recall a few notions required for discussing optimization on Riemannian manifolds, specialized to the case of the sphere, $\mathcal{S}_{n-1}$. For any $z\in\mathcal{S}_{n-1}$ the {\em tangent space} is $T_{z}\mathcal{S}_{n-1} = \{v \in \bbR^n: \ v^{\top}z = 0\}$. Let $\operatorname{Proj}_z$ denote the projection onto $T_{z}\mathcal{S}_{n-1}$. One can verify that $\operatorname{Proj}_z(w) = w - (w^{\top}z)z$. For any smooth $g:\mathcal{S}_{n-1} \to \bbR$ the {\em Riemannian gradient} at $z\in\mathcal{S}_{n-1}$ is the projection of the regular gradient onto $T_{z}\mathcal{S}_{n-1}$: $\operatorname{grad}_z g = \operatorname{Proj}_z\nabla g(z)$. Similarly, we may define the {\em Riemannian Hessian} at $z \in \mathcal{S}_{n-1}$ as the operator $\operatorname{Hess}g(z) = \operatorname{Proj}_z\circ \left(\nabla^2g(z) - \nabla g(z)^{\top}z\right)\circ \operatorname{Proj}_z$; see \citep[Sec. 7]{boumal2020introduction} for further details. \\

Given $z\in \mathcal{S}_{n-1}$ and $v\in T_{z}\mathcal{S}_{n-1}$ with $\|v\| = 1$, there exists a unique {\em geodesic} emanating from $z$ in the direction of $v$. Geogescis generalize the role of the straight lines in Euclidean geometry. On $\mathcal{S}_{n-1}$ the geodesics are also known as the great circles.

Define $\gamma_{z,v}(t): \mathbb{R} \to \mathcal{S}_{n-1}$ as the geodesic mapping at $z$ with direction $v\in T_{z}\mathcal{S}_{n-1}$ and the {\em exponential map} at $z\in \mathcal{S}_{n-1}$:
\begin{align*}
    \operatorname{exp}_{z}: & T_{z}\mathcal{S}_{n-1} \to \mathcal{S}_{n-1} \\
    & v \mapsto \gamma_{z,\hat{v}}(\|v\|) \text{ where } \hat{v} := v\big/\|v\|
\end{align*}

Riemannian Gradient Descent or RGD mimics regular ({\em i.e.} Euclidean) gradient descent, except instead of using $\nabla g(z_k)$ we use $\operatorname{grad}g(z_k)$, and instead of stepping along the straight line in the direction $-\nabla g(z_k)$, we move along the geodesic in the direction $-\operatorname{grad}g(z_k)$:
\begin{equation}
    z_{k+1} = \operatorname{exp}_{z_k}(-\alpha_k\operatorname{grad}g(z_k))
\end{equation}
By construction, $z_{k+1} \in \mathcal{S}_{n-1}$, so RGD is a {\em feasible algorithm}. Algorithm~\ref{alg:HadRGD}, which we dub HadRGD, describes how to apply RGD to \eqref{C_1} via the Hadamard parametrization. 
The converted objective $g$ inherits the smoothness of $f$:

\begin{lemma}
If $f$ is $L$-Lipschitz differentiable, then $g$ is $\tilde{L}$-Lipschitz differentiable with $\tilde{L} = 4L + 2M$ where $M = \max_{x\in \Delta_n}\|\nabla f(x)\|_{\infty}$.
\label{lemma:Lipschitz}
\end{lemma}
See Appendix~\ref{app:RiemGeom} for the proof. As $\nabla f(x)$ is continuous and $\Delta_n$ is compact, $M < \infty$.

\section{Landscape Analysis}
\label{sec:LandscapeAnalysis}

We use the first- and second-order KKT conditions to establish the one-to-one correspondence between second-order KKT points of \eqref{C_1} and \eqref{NC_1}. 

\paragraph{First-order KKT conditions on $x^\star$}
The Lagrangian of \eqref{C_1} is given by
\begin{align}
     \mathcal{L}_C(x,\lambda,\beta) = f(x) - \lambda (1_n^\top x-1) - \beta^\top x 
\end{align}
where $\lambda\in\R,\beta\in\R^n$ are Lagrangian multipliers.
The first-order KKT conditions of \eqref{C_1} are: there exist $\lambda^\star\in\mathbb{R}$ and $\beta^\star\in\mathbb{R}^n$ such that 
\begin{subequations}
\begin{align}
\nabla f(x^\star) &= \lambda^\star 1_n + \beta^\star \label{KKT:C1:1}
 \\
 x^\star& \ge 0
 \label{KKT:C1:2}\\
 \beta^\star &\ge 0 \label{KKT:C1:3}
 \\
 x^\star\odot \beta^\star &=0 \label{KKT:C1:4}
 \\
 1_n^\top x^\star &=1 \label{KKT:C1:5}
\end{align}
\label{KKT:C1}
\end{subequations}

\paragraph{Second-order KKT conditions on $x^\star$}
As
\begin{equation}
    \nabla^2_x\mathcal{L}_{C}(x,\lambda,\beta) = \nabla^2_xf(x) 
\end{equation}
the second-order KKT conditions here are simply
\begin{equation}
    u^\top\nabla^2f(x^\star)u \geq 0 
    \label{eq:KKT:C1:2}
\end{equation}
for all $u$ satisfying $1_n^\top u = 0$ and $[u]_i = 0$ if $[x^\star]_i = 0$.

\paragraph{First-order KKT conditions on $z^\star$}
The Lagrangian of \eqref{NC_1} is
\begin{align}
    \mathcal{L}_N(z,\lambda_N) = g(z) - \lambda_N (\|z\|_2^2-1)
\end{align}
The first-order KKT conditions of \eqref{NC_1} are: there exists $\lambda_N^\star\in\mathbb{R}$ such that 
\begin{subequations}
\begin{align}
\nabla f(z^\star\odot z^\star)\odot z^\star &= \lambda^\star_N z^\star \label{KKT:NC1:1}
 \\
\|z^\star\|^2_{2} &=1 \label{KKT:NC1:2}
\end{align}
\label{KKT:NC1}
\end{subequations}

\paragraph{Second-order KKT conditions on $z^\star$}
Using \ref{H5}, the second-order KKT conditions of \eqref{NC_1} are: there exists $\lambda_N^\star\in\mathbb{R}$ such that $(z^\star,\lambda^\star_N)$ satisfies \eqref{KKT:NC1:1} and \eqref{KKT:NC1:2} and for all $d \bot z^\star$,
\begin{align}
&d^{\top}\left[\nabla_z^2 \mathcal{L}_N(z^\star,\lambda_N^\star)\right]d\ge 0 
\nonumber\\
\iff&
d^{\top}\bigg[2\mathrm{diag}(\nabla f(z^\star\odot z^\star)) +4\mathrm{diag}(z^\star) \nabla^2 f(z^\star\odot z^\star) \mathrm{diag}(z^\star) - 2\lambda_N^\star I_n \bigg]d \ge 0
\nonumber\\
\overset{\text{\ref{H3}}}{\iff}&
 2\langle\nabla f(z^\star\odot z^\star), d\odot d\rangle  - 2\lambda_N^\star \|d\|^2
+  4 (z^\star\odot d)^\top \nabla^2 f(z^\star\odot z^\star)(z^\star\odot d) \ge 0 
\label{KKT2:NC1}
\end{align}
where $\nabla_z^2(\cdot)$ denotes the partial Hessian with respect to $z$. \\


\paragraph{Strict saddle points} We say $x^\star$ is a strict saddle point of \eqref{C_1} if it satisfies the first-order KKT conditions \eqref{KKT:C1:1}--\eqref{KKT:C1:5} but there exists a $u\in \mathbb{R}^n$ satisfying $1_n^\top u = 0$ and $[u]_i = 0$ if $[x^\star]_i = 0$ and
\begin{equation}
    u^\top\nabla^2f(x^\star)u < 0
\end{equation}
Similarly, we say $z^\star$ is a strict saddle point of \eqref{NC_1} if it satisfies the first-order KKT conditions \eqref{KKT:NC1:1} and \eqref{KKT:NC1:2} yet there exists $d\bot z^\star$ such that
\begin{equation}
d^\top \left[\nabla^2_z \mathcal{L}_N(z^\star,\lambda^\star_N)\right] d<0.  \label{eqn:strict:saddle} 
\end{equation}
Note that any first-order KKT point must be either a strict saddle point or a second-order KKT point. In particular, when every second-order KKT point is a local minimizer, we say the problem satisfies the {\em strict saddle property}. 

We now provide our main landscape analysis results.
\begin{theorem}
\label{thm:H1}
\begin{enumerate}
    \item Suppose that $x^\star$ is a second-order KKT point of \eqref{C_1}. Then, all $z^\star$ satisfying $z^\star\odot z^\star = x^\star$ are second-order KKT points of \eqref{NC_1}.
    
    \item Conversely, suppose $z^\star$ is a second-order KKT point of \eqref{NC_1}. Then, $x^\star = z^\star\odot z^\star$ is a second-order KKT point of \eqref{C_1}.
\end{enumerate}
\end{theorem}


\begin{proof}

\begin{enumerate} 
    \item Suppose $z^\star$ satisfies $z^\star\odot z^\star = x^\star$. We first show $z^\star$ satisfies the first-order KKT conditions \eqref{KKT:NC1:1} and \eqref{KKT:NC1:2}. Multiplying both sides of the optimality condition \eqref{KKT:C1:1} by $z^\star$:
    \begin{align}
        & \nabla f(x^{\star})\odot z^{\star} = \lambda^{\star}1_n\odot z^{\star} + \beta^{\star}\odot z^{\star} \\
   \overset{\text{\ref{H1}}}{\Longrightarrow} &  \nabla f(z^{\star}\odot z^{\star})\odot z^{\star} = \lambda^{\star} z^{\star} + \beta^{\star}\odot z^{\star}   \label{eq:Reduce_to_KKT1_NVX}
    \end{align}
    By complementary slackness \eqref{KKT:C1:4}: $z^\star\odot z^\star\odot\beta^\star=0$, hence $z^\star\odot\beta^\star=0$ by \ref{H2}. Thus \eqref{eq:Reduce_to_KKT1_NVX} reduces to \eqref{KKT:NC1:1} by choosing $\lambda_N^\star=\lambda^\star$. Note \eqref{KKT:C1:5} is equivalent to \eqref{KKT:NC1:2} as
    \[
        1 = 1_n^{\top}x^{\star} = 1_n^{\top}(z^{\star}\odot z^{\star}) = \sum_{i=1}^{n}(z_i^{\star})^2 = \|z^{\star}\|_2^{2}
    \]
    
    It remains to show the second-order KKT conditions \eqref{KKT2:NC1} of \eqref{NC_1} with $\lambda_N^\star=\lambda^\star$. Since $x^\star$ satisfies \eqref{KKT:C1:1} and \eqref{KKT:C1:3}, we have
    \begin{align}
        & \nabla f(x^\star) = \lambda^\star 1_n + \beta^\star ,\ \beta^{\star}\ge 0
       \\ 
       \iff &  \langle \nabla f(z^\star\odot z^\star), d\odot d\rangle - \langle \lambda^\star 1_n, d\odot d \rangle = \langle \beta^\star, d\odot d \rangle \text{ as } x^\star = z^\star\odot z^\star \\
       \iff &  \langle \nabla f(z^\star\odot z^\star), d\odot d\rangle - \langle \lambda^\star 1_n, d\odot d \rangle = \sum_i [\beta^\star]_i[d]_i^2 \geq 0 \text{ as } \beta^\star \geq 0 \label{eq:With_beta} \\
       \iff & \langle \nabla f(z^\star\odot z^\star), d\odot d\rangle - \lambda^\star \|d\|^2 \ge 0,\ \forall d 
        \label{gradPSD}
    \end{align}

    Plugging this into \eqref{KKT2:NC1}, we get for any $d\in\mathbb{R}^n$ satisfying $d \perp z^\star$
    \begin{align}
     & 2\langle\nabla f(z^\star\odot z^\star), d\odot d\rangle - 2\lambda^\star \|d\|^2 + 4(z^\star\odot d)^\top \nabla^2 f(z^\star\odot z^\star)(z^\star\odot d) \nonumber \\
     & \ge 4(z^\star\odot d)^\top \nabla^2 f(z^\star\odot z^\star)(z^\star\odot d)
     \label{H1:PSD}
    \end{align}
    We now show $(z^\star\odot d)^\top \nabla^2 f(z^\star\odot z^\star)(z^\star\odot d) \geq 0$. Define $u = z^\star\odot d$ and observe
    \begin{equation}
        1_n^{\top}u = \sum_{i=1}^n [u]_i = \sum_{i=1}^n [z^\star]_i[d]_i = \langle z^\star,d\rangle = 0 \quad \text{ as } d \perp z^\star
    \end{equation}
    Moreover, if $[x^\star]_i = 0$ then $[z^\star]_i = 0$ and so $[u]_i = 0$. Appealing to \eqref{eq:KKT:C1:2} we observe $u^\top\nabla^2f(x^\star)u \geq 0$. Returning to \eqref{H1:PSD} we conclude that the second-order optimality condition \eqref{KKT2:NC1} indeed holds. 
    
  \item We show $x^\star = z^\star \odot z^\star$ is a second-order KKT point of \eqref{C_1}. From \eqref{KKT:NC1:2} we deduce
    \begin{equation}
       x^\star \ge 0 \text{ and } 1_n^\top x^\star = 1. \label{eq:KKT_1_temp}
    \end{equation}
    That is, $x^\star$ satisfies the KKT conditions \eqref{KKT:C1:2} and \eqref{KKT:C1:5} of \eqref{C_1}. Using the relation
    \begin{equation}
        \lambda_N^\star \|d\|^2 = \lambda_N\sum_{i=1}^n[d]_i^2 = \lambda_N\sum_{i=1}^n[d\odot d]_i = \langle \lambda_N 1_n,d\odot d\rangle
    \end{equation}
    we rewrite the second-order KKT condition \eqref{KKT2:NC1} as
   \begin{align}
       0 &\leq 4(z^\star\odot d)^\top \nabla^2 f(x^\star)(z^\star\odot d) +2\langle\nabla f(x^\star), d\odot d\rangle - 2\lambda_N^\star \|d\|^2 \\
       & = \underbrace{4(z^\star\odot d)^\top \nabla^2 f(x^\star)(z^\star\odot d)}_{=A} +\underbrace{2\langle\nabla f(x^\star) - \lambda_N^\star 1_n, d\odot d\rangle}_{=B} \label{eq:A_and_B}
   \end{align}
    for any $d \perp z^\star$. Since $z^\star$ satisfies \eqref{KKT:NC1:1}, there exists $\lambda_N^\star\in\mathbb{R}$ such that
   \begin{align}
       & \nabla f(x^\star) \odot z^\star = \lambda_N^\star z^\star \\
       \Rightarrow & [\nabla f(x^\star)]_k[z^{\star}]_k = \lambda_N^\star [z^\star]_k \\
       \Rightarrow & [\nabla f(x^\star)]_k = \lambda_N^\star \text{ for all } k \text{ with } [z^\star]_k \neq 0 \\
       \Rightarrow & [\nabla f(x^\star)]_k - \lambda_N^\star = 0 \text{ for all } k \text{ with } [x^\star]_k \neq 0 \label{eq:beta_setup} 
   \end{align}
    Define $\beta^\star = \nabla f(x^\star) - \lambda_{N}^\star 1_n$. We show $\beta^\star$ is non-negative (the optimality condition \eqref{KKT:C1:3}). Indeed, we already have that  $[\beta^\star]_k = 0$ when $[x^\star]_k \neq 0$ (equivalently $[z^\star]_k \neq 0$) by \eqref{eq:beta_setup}. Now suppose, for the sake of contradiction, there exists an index $k^\star$ such that $[z^\star]_{k^\star} = 0$ and $[\beta^\star]_{k^\star} < 0$. Let $d = e_{k^\star}$. Observe that $d\perp z^\star$ and $z^\star\odot d = 0$. Substituting this $d$ into \eqref{eq:A_and_B} we observe that $A=0$ and $B < 0$. This contradicts the fact that $A + B \geq 0$, hence such an index $k^\star$ cannot exist and so $\beta^\star \geq 0$. Taking $\lambda^\star = \lambda_N^\star$ and rearranging the definition of $\beta^\star$ yields \eqref{KKT:C1:1}. Finally, complementary slackness \eqref{KKT:C1:4} holds as $[\beta^\star]_k = 0$ whenever $[x^\star]_k > 0$, by \eqref{eq:beta_setup}. So, $x^\star$ satisfies all of the first-order KKT conditions. \\
   
   We now show $x^\star$ also satisfies the second-order optimality condition. So, consider any $u \in \mathbb{R}^n$ satisfying $1_n^\top u = 0$ and $[u]_i = 0$ for all $[x^\star]_i = 0$ (equivalently $[z^\star]_i = 0$). Define a corresponding vector $d \in \mathbb{R}^n$ as
   \begin{equation}
       [d]_i = \left\{\begin{array}{cc} \frac{[u]_i}{[z^\star]_i} & \text{ if } [z^\star]_i \neq 0 \\ 0 & \text{ otherwise } \end{array}\right. \label{eq:construct_d}
   \end{equation}
 Then $z^\star \perp d$ as
   \begin{align}
        \langle d, z^\star \rangle &= \sum_{i=1}^n [d]_i[z^\star]_i \\
        &= \sum_{i: [z^\star]_i = 0}0 + \sum_{i: [z^\star]_i \neq 0}\left(\frac{[u]_i}{[z^\star]_i}\right)[z^\star]_i \\
        & = \sum_{i: [z^\star]_i = 0} [u]_i + \sum_{i: [z^\star]_i \neq 0} [u]_i \quad \text{ as } [z^\star]_i = 0 \Rightarrow [u^\star]_i = 0 \\
        & = \sum_{i=1}^n [u]_i = 1_n^\top u = 0.
   \end{align}
  Appealing once more to \eqref{eq:A_and_B} and recalling that $[\beta^\star]_i = 0$ for all $i$ with $[z^\star]_i \neq 0$ \eqref{eq:beta_setup}: 
   \begin{align}
       B &= 2\langle \beta^\star, d\odot d\rangle = 2\sum_{i: [z^\star]_i \neq 0}[\beta^{\star}]_i\left(\frac{[u]_i}{[z^\star]_i}\right)^2 = 0 \\
       A &= 4u^\top \nabla^2 f(x^\star) u, \label{eq:Simplify_A}
   \end{align}
   where \eqref{eq:Simplify_A} follows from the fact that $z^\star\odot d = u$ as $[z]_i = 0 \Rightarrow [u]_i = 0$. So, from \eqref{eq:A_and_B} we deduce
   \begin{equation}
       0 \leq A + B  = 4u^\top \nabla^2 f(z^\star\odot z^\star) u
   \end{equation}
   {\em i.e.} the second-order optimality condition \eqref{eq:KKT:C1:2} holds. Thus $x^\star = z^\star\odot z^\star$ is indeed a second-order KKT point of \eqref{C_1}.
\end{enumerate} 
\end{proof}

We now deduce a series of corollaries:
\begin{corollary}
Problem \eqref{C_1} has the strict saddle property if and only if \eqref{NC_1} has the strict saddle property.
\end{corollary}

\begin{proof}
First, suppose that \eqref{C_1} has the strict saddle property. If $z^\star$ is a second-order KKT point, then so is $x^\star = z^\star\odot z^\star$, by Theorem~\ref{thm:H1}. Hence, $x^\star$ is a local minimizer, and so choosing $\varepsilon$ sufficiently small we have that
\begin{equation}
    f(x^\star) \leq f(x) \quad \text{ for all } x \text{ with } \left|[x]_i - [x^\star]_i\right| < \varepsilon \ i=1,\ldots, n. \label{eq:x_local_min}
\end{equation}
Now, for any $z$ with $|[z]_i - [z^\star]_i| < \varepsilon/2$ for all $i=1,\ldots, n$ define $x = z\odot z$ and observe that
\begin{equation}
    \left|[x]_i - [x^\star]_i\right| = \left|[z]_i^2 - [z^\star]_i^2\right| = \left|[z]_i - [z^\star]_i\right|\left|[z]_i + [z^\star]_i\right| \leq 2\left|[z]_i - [z^\star]_i\right| < \varepsilon
\end{equation}
whence from \eqref{eq:x_local_min} we obtain
\begin{equation}
    g(z^\star) = f(x^\star) \leq f(x) = g(z),
\end{equation}
{\em i.e.} $z^\star$ is a local minimizer. The proof of the converse is similar.
\end{proof}

\begin{corollary}
Suppose $f$ in \eqref{C_1} is convex. Then $z^\star$ is a second-order KKT point of \eqref{NC_1} if and only if $x^\star = z^\star\odot z^\star$ is a global minimizer of \eqref{C_1}
\end{corollary}

\begin{proof}
This follows immediately from the fact that $x^\star$ is a global minimizer of \eqref{C_1} if and only if $x^\star$ is a second-order (in fact, first-order suffices) KKT point.
\end{proof}

\section{Non-degeneracy}

We define a notion of strict, or non-degenerate, KKT points.

\begin{definition}
  We say $x^\star$ is a non-degenerate second-order KKT point of \eqref{C_1} if there exist $\lambda^\star \in \mathbb{R}$ and $\beta^\star \in \mathbb{R}^n$ such that the first-order KKT conditions \eqref{KKT:C1:1}--\eqref{KKT:C1:3} and \eqref{KKT:C1:5} are satisfied, {\em strict complementary slackness} holds:
  \begin{equation}
      x^\star\odot \beta^\star = 0 \text{ and } \beta^\star_j > 0 \text{ for all } j \text{ with } x^\star_j = 0 \tag{8d$^\prime$},
  \end{equation}
  and the second-order KKT condition \eqref{eq:KKT:C1:2} holds {\em with strict inequality} for all $u \neq 0$ satisfying $1_n^\top u = 0$ and $[u]_i = 0$ if $[x^\star]_i = 0$.
  \label{Def:non-degen_KKT}
\end{definition}

By \cite[Prop.4.3.2]{bertsekas1997nonlinear} all non-degenerate second-order KKT points are strict ({\em i.e.} isolated) local minima. We remark that the strict complementary slackness is necessary for this characterization (see \cite[Example 4.3.2]{bertsekas1997nonlinear}. 

\begin{definition}
  We say $z^\star$ is a non-degenerate KKT point of \eqref{NC_1} if there exists $\lambda_N^\star$ such that the first-order KKT conditions \eqref{KKT:NC1:1} and \eqref{KKT:NC1:2} are satisfied and the second-order KKT condition \eqref{KKT2:NC1} holds {\em with strict inequality} for all $d \neq 0$ satisfying $d\bot z^\star$.
\end{definition}

As for the case of (non-strict) second-order KKT points, we can compare non-degenerate second-order KKT points for \eqref{C_1} and \eqref{NC_1}.

\begin{theorem}
    Suppose that $x^\star$ is a non-degenerate second-order KKT point of \eqref{C_1}. Then, all $z^\star$ satisfying $z^\star\odot z^\star = x^\star$ are non-degenerate second-order KKT points of \eqref{NC_1}. 
    \label{thm:Non-degen_equiv}
\end{theorem}

Before proving this theorem, let us discuss why it cannot be deduced straightforwardly from the proof of Theorem~\ref{thm:H1}. From \eqref{H1:PSD} and \eqref{KKT2:NC1} we deduce that 
\begin{equation}
    d^{\top}\left[\nabla_z^2 \mathcal{L}_N(z^\star,\lambda_N^\star)\right]d \geq 4(z^\star\odot d)^\top \nabla^2 f(z^\star\odot z^\star)(z^\star\odot d) \geq 0
\end{equation}
where $u := z^\star\odot d$ satisfies $1_n^\top u = 0$ and $[u]_i = 0$ if $[x^\star]_i = 0$. However, it may be the case that $u = 0$, even when $d\neq 0$, if $[z^\star]_i = 0$ whenever $[d]_i \neq 0$. Hence, we cannot naively conclude from the positive-definiteness of $\nabla^2f(x^\star)$ that $z^\star$ is non-degenerate. Rather, the strict complementary slackness condition is required, as the following proof illustrates.

\begin{proof}
Suppose $x^\star$ is a non-degenerate second-order KKT point of \eqref{C_1}. From \eqref{KKT2:NC1} and \eqref{eq:With_beta} we deduce
\begin{align}
    d^{\top}\left[\nabla_z^2 \mathcal{L}_N(z^\star,\lambda_N^\star)\right]d & \geq \sum_i [\beta^\star]_i[d]_i^2 + 4 u^\top \nabla^2f(x^\star) u \geq 0
    \label{eq:non-degen_KKT_beta}
\end{align}
where $u := z^\star \odot d$. Now, suppose $u^\top \nabla^2f(x^\star) u = 0$. As $x^\star$ is non-degenerate, this can only happen if $u = 0$, {\em i.e.} if $[z^\star]_i = 0$ whenever $[d]_i \neq 0$. But, by strict complementary slackness, $[\beta^\star]_i > 0$ whenever $[z^\star]_i = 0$, as $[z^\star]_i = 0 \Rightarrow [x^\star]_i = 0$. Hence, in this case we have $\sum_i [\beta^\star]_i[d]_i^2 > 0$. It follows that \eqref{eq:non-degen_KKT_beta} holds with strict inequality for all $d$ with $d\bot z^\star$ and $d \neq 0$. 
\end{proof}

\section{Perturbed Riemannian Gradient Descent}
In this section we study Problem \eqref{NC_1} from the Riemannian perspective. Specifically, we regard $g(z)$ as a function on the manifold $\mathcal{S}_{n-1}$, whence \eqref{NC_1} becomes an unconstrained, Riemannian optimization problem. Recalling notation from Section~\ref{sec:Riem_Opt},

\begin{definition}
  $z^\star$ is a second-order stationary point of $g: \mathcal{S}_{n-1} \to \mathbb{R}$ if 
  \begin{equation}
      \operatorname{grad} g(z^\star) = 0
~\text{ and }~
\lambda_{\min }(\operatorname{Hess} g(z)) \geq 0
  \end{equation}
  where $\lambda_{\min }(H)$ denotes the smallest eigenvalue of the symmetric operator $H$. We say $z^\star$ is a {\em non-degenerate} second-order stationary point of $g: \mathcal{S}_{n-1} \to \mathbb{R}$ if in addition $\lambda_{\min }(\operatorname{Hess} g(z)) > 0$.
  \label{def:SOSP}
\end{definition}

Note that $z^{\star}$ is a second-order stationary point in the Riemannian sense if and only if $z^{\star}$ is a second-order KKT point of Problem \eqref{NC_1} \citep{luenberger1972gradient,boumal2020introduction}. This justifies using these terms interchangeably, although for clarity we will always use the terminology ``stationary point'' when viewing Problem \eqref{NC_1} as a Riemannian optimization problem, and ``KKT point'' when viewing \eqref{NC_1} as a constrained optimization problem. 

\begin{definition}[\cite{criscitiello2019efficiently}]
  A point $z \in \mathcal{S}_{n-1}$ is an $\epsilon$-second-order stationary point of the twice-differentiable function $g: \mathcal{S}_{n-1} \rightarrow \mathbb{R}$ if
$$\|\operatorname{grad} g(z)\| \leq \epsilon
~\text{ and }~
\lambda_{\min }(\operatorname{Hess} g(z)) \geq-\sqrt{\rho \epsilon}
$$
where $\rho$ denotes the Lipschitz constant of the Hessian of the pullback of $g$ from the manifold to tangent space.
\label{def:EpsSecondOrder}
\end{definition}

There is no guarantee RGD applied to a nonconvex function will converge to a second-order stationary point (it may find a saddle point). Fortunately, \citep{criscitiello2019efficiently} shows a {\em Perturbed} version of RGD (PRGD, c.f. Algorithm 1 in \citep{criscitiello2019efficiently}) will, with high probability (w.h.p), find an $\epsilon$-second-order KKT point. We reproduce their result, adapted to the sphere, here.
\begin{theorem}[\cite{criscitiello2019efficiently}]
Let $\{z_k\}_{k=1}^K$ be the sequence of iterates generated by PRGD applied to $g: \mathcal{S}_{n-1} \to \bbR$ for $K$ iterations. If $K = \mathcal{O}\left((\log n)^{4} / \epsilon^{2}\right)$ then $\{z_k\}_{k=1}^K$ contains an $\epsilon$-second-order stationary point of $g(z)$ w.h.p.
\label{thm:Criscitiello}
\end{theorem}

We call the combination of the Hadamard parametrization and PRGD {\em HadPRGD} (see Algorithm~\ref{alg:HadPRGD}). As a consequence of our landscape analysis:

\begin{theorem}
Let $\{x_k\}_{k=1}^{\infty}$ be the sequence of iterates produced by HadPRGD. Then w.h.p. $\{x_k\}_{k=1}^{\infty}$ contains a subsequence converging to a second-order KKT point of \eqref{C_1}: $x_{k_{\ell}} \to x^\star$.
\label{thm:Had_PRGD_convergence}
\end{theorem}

For the sake of modularity, before proving Theorem~\ref{thm:Had_PRGD_convergence} we first prove a lemma.

\begin{lemma}
Suppose $g$ is twice continuously differentiable. Then, for any $\delta > 0$ there exists an $\epsilon_{\delta} > 0$ such that:
\begin{equation}
    z \text{ is an $\epsilon_{\delta}$-second-order stationary point } \Rightarrow \|z - z^{\star}\| < \delta
\end{equation}
where $z^{\star}$ is a second-order stationary point. 
\label{lemma: converge to stationary}
\end{lemma}

\begin{proof}
Fix $\delta > 0$ and suppose to the contrary that no such $\epsilon_{\delta}$ exists. Then, for any $\epsilon_k := 1/k$ we may find an $\epsilon_k$-second-order stationary point $z_k$ such that:
\begin{equation}
\|z_k - z^{\star}\| \geq \delta \text{ for all second-order stationary points $z^\star$.}
\label{eq: Second Order Contradiction}
\end{equation}
As $\mathcal{S}_{n-1}$ is compact, passing to a subsequence if necessary we may assume $z_k$ converges: $\lim_{k\to\infty} z_k = \tilde{z} \in \mathcal{S}_{n-1}$. By continuity, $\tilde{z}$ is a second-order stationary point:
\begin{align*}
    &  \|\operatorname{grad} g(\tilde{z})\| = \|\operatorname{grad} g(\lim_{k\to\infty}z_k)\| = \lim_{k\to\infty}\|\operatorname{grad} g(z_k)\| \leq \lim_{k\to\infty} \epsilon_k = 0 \\
    & \lambda_{\min }(\operatorname{Hess} g(\tilde{z})) = \lambda_{\min }(\operatorname{Hess} g(\lim_{k\to\infty}z_k)) = \lim_{k\to\infty}\lambda_{\min }(\operatorname{Hess} g(z_k)) \geq \lim_{k\to\infty}-\sqrt{\rho\epsilon_k}  = 0
\end{align*}
As $\tilde{z} = \lim_{k\to\infty} z_k$ this contradicts \eqref{eq: Second Order Contradiction}.
\end{proof}

\begin{proof}[Proof of Theorem~\ref{thm:Had_PRGD_convergence}]
Let $\{z_k\}_{k=1}^{\infty}$ denote the auxiliary sequence generated by HadPRGD. Defining $\epsilon_{\ell} := 1/\ell$ by Theorem~\ref{thm:Criscitiello} we know there exists $T_{\ell}$ such that $\{z_k\}_{k=1}^{T_{\ell}}$ contains an $\epsilon_{\ell}$-second-order stationary point w.h.p, call it $z_{k_{\ell}}$. As $\epsilon_{\ell}\to 0$, by Lemma~\ref{lemma: converge to stationary} $z_{k_{\ell}} \to z^\star$, a second-order stationary point of Problem \eqref{NC_1}. By \cite{luenberger1972gradient} $z^\star$ is equivalently a second-order KKT point of \eqref{NC_1}. The sequence $\{x_{k_{\ell}}:= z_{k_{\ell}}\odot z_{k_{\ell}}\}_{\ell=1}^{\infty}$ converges to $x^\star = z^\star\odot z^\star$, which by Theorem~\ref{thm:H1} is a second-order KKT point of \eqref{C_1}.
\end{proof}

Providing generic, quantitative convergence guarantees is non-trivial, as one has to relate $\epsilon$-second-order stationary points of $g$ to $\epsilon$-second-order KKT points of \eqref{C_1}, appropriately defined (see, for example \citep{dutta2013approximate}). We leave this for future work and instead provide a convergence rate for the following special case.

\begin{theorem}
Suppose problem \eqref{C_1} has the strict saddle property, has all local minimizers also global minimizers, and has that all second-order KKT points are non-degenerate. Then with high probability HadPRGD finds $x_k$ satisfying
\begin{equation}
    f(x_k) - \min_{x\in\Delta_n}f(x) \leq \epsilon
\end{equation}
within $K = \mathcal{O}((\log n)^4/\epsilon)$ iterations, for $\epsilon$ small enough. 
\end{theorem}

\begin{proof}
Fix $z_0$. As in Theorem~\ref{thm:Had_PRGD_convergence} we know there exists a subsequence $z_{k_{\ell}} \to z^\star$, a second-order stationary point of $g$, and $x_{k_{\ell}} \to x^\star := z^\star\odot z^\star$, a second-order KKT point of \eqref{C_1}. As \eqref{C_1} is strict-saddle, $x^\star$ is a local minimizer, and hence a non-degenerate global minimizer (by assumption). From Theorem~\ref{thm:Non-degen_equiv} $z^\star$ is a non-degenerate second-order KKT point of \eqref{NC_1}, and so again by \cite{luenberger1972gradient} (see also \cite{boumal2020introduction}) $z^\star$ is a non-degenerate second-order stationary point of $g: \mathcal{S}_{n}\to\mathbb{R}$. That is, $\operatorname{Hess} g(z^{\star})$ is positive definite. By continuity there exists a geodesic ball $B(z^{\star}, \delta)$ upon which $\operatorname{Hess} g(z)$ is positive definite, {\em i.e.} $g$ restricted to $B(z^{\star}, \delta)$ is geodesically $\tau$-strongly convex for some $\tau > 0$ and so satisfies the Polyak-\L ojasiewicz (PL) inequality \cite[Lemma 11.28]{boumal2020introduction}:
\begin{equation}
  g(z) - g(z^{\star}) \leq \frac{1}{2\tau}\|\operatorname{grad} g(z)\|^2  
\end{equation}

From Lemma~\ref{lemma: converge to stationary} there exists $\epsilon_{\delta} > 0$ such that if $z$ is an $\epsilon_{\delta}$-second-order KKT point then $z \in B(z^{\star}, \delta)$. So, suppose $\epsilon >0$ is small enough that $\sqrt{2\tau\epsilon} < \epsilon_\delta$. By Theorem~\ref{thm:Criscitiello} HadPRGD finds a $\sqrt{2\tau\epsilon}$-second-order KKT point, call it $z_{\sqrt{2\tau\epsilon}}$, within $K = \mathcal{O}\left((\log n)^4/(\sqrt{2\tau\epsilon})^2\right) = \mathcal{O}\left((\log n)^4/\epsilon\right)$ iterations (w.h.p). As $\sqrt{2\tau\epsilon} < \epsilon_\delta$ we know $z_{\sqrt{2\tau\epsilon}}$ is also an $\epsilon_{\delta}$-second-order KKT point hence $z_{\sqrt{2\tau\epsilon}} \in B(z^{\star}, \delta)$. Appealing to the PL inequality and letting $x_{\sqrt{2\tau\epsilon}} = z_{\sqrt{2\tau\epsilon}}\odot z_{\sqrt{2\tau\epsilon}}$:

\begin{equation}
    f(x_{\sqrt{2\tau\epsilon}}) - \min_{x\in\Delta_n}f(x) = f(x_{\sqrt{2\tau\epsilon}}) - f(x^{\star}) = g(z_{\sqrt{2\tau\epsilon}}) - g(z^{\star}) \leq \frac{1}{2\tau} \|\operatorname{grad} g(z_{\sqrt{2\tau\epsilon}})\|^2 \leq \frac{1}{2\tau} \left(\sqrt{2\tau\epsilon}\right)^2 \leq \epsilon
\end{equation}

\end{proof}

\section{Selecting the Step Size}
Experimentally that HadPRGD tolerates much larger step-sizes than the Lipschitz constant bound in Lemma~\ref{lemma:Lipschitz} suggests, leading to faster convergence in practice. To substantiate this we consider the simplex-constrained under-determined least squares problem
\begin{equation}
    \operatorname*{minimize}_{x\in\Delta_n} \left\{f(x) = \|Ax - b\|_2^2\right\}
    \label{eq:Underdetermined_Least_Squares}
\end{equation}
where $A\in \bbR^{m\times n}$ with $m = 0.1n$ and $b = Ax_{\mathrm{true}}$ for randomly selected $x_{\mathrm{true}} \in \mathrm{int}(\Delta_n)$. We record, as a function of $n$, the number of iterations and wall-clock time required by HadPRGD and PGD to reach  to find an $\varepsilon$-optimal solution with $\varepsilon=10^{-16}$ (see Figure~\ref{fig:LeastSquares_Interior}). The step sizes of HadPRGD and PGD are hand-tuned to be as large as possible while still converging.

\begin{algorithm}[tb]
\caption{HadRGD for \eqref{C_1}}
\label{alg:HadRGD}
\textbf{Input}: $x_0 \in \Delta_n$: initial point, $\alpha$: step size, $K$: number of iterations, $f$: original objective function
\begin{algorithmic}[1] 
\STATE $z_0 = \sqrt{x_0}$ \algorithmiccomment{(Defined componentwise)}
\STATE $g(z) := f(z\odot z)$
\FOR{k=1,\ldots, K}
\STATE $z_{k+1} = \operatorname{exp}_{x_k}(-\alpha\operatorname{grad}g(z_k))$
\ENDFOR
\end{algorithmic}
\textbf{Return} $x_{K} = z_{K}\odot z_{K}$
\end{algorithm}

\begin{figure*}[!htp]
    \centering
    \includegraphics[width=0.47\linewidth]{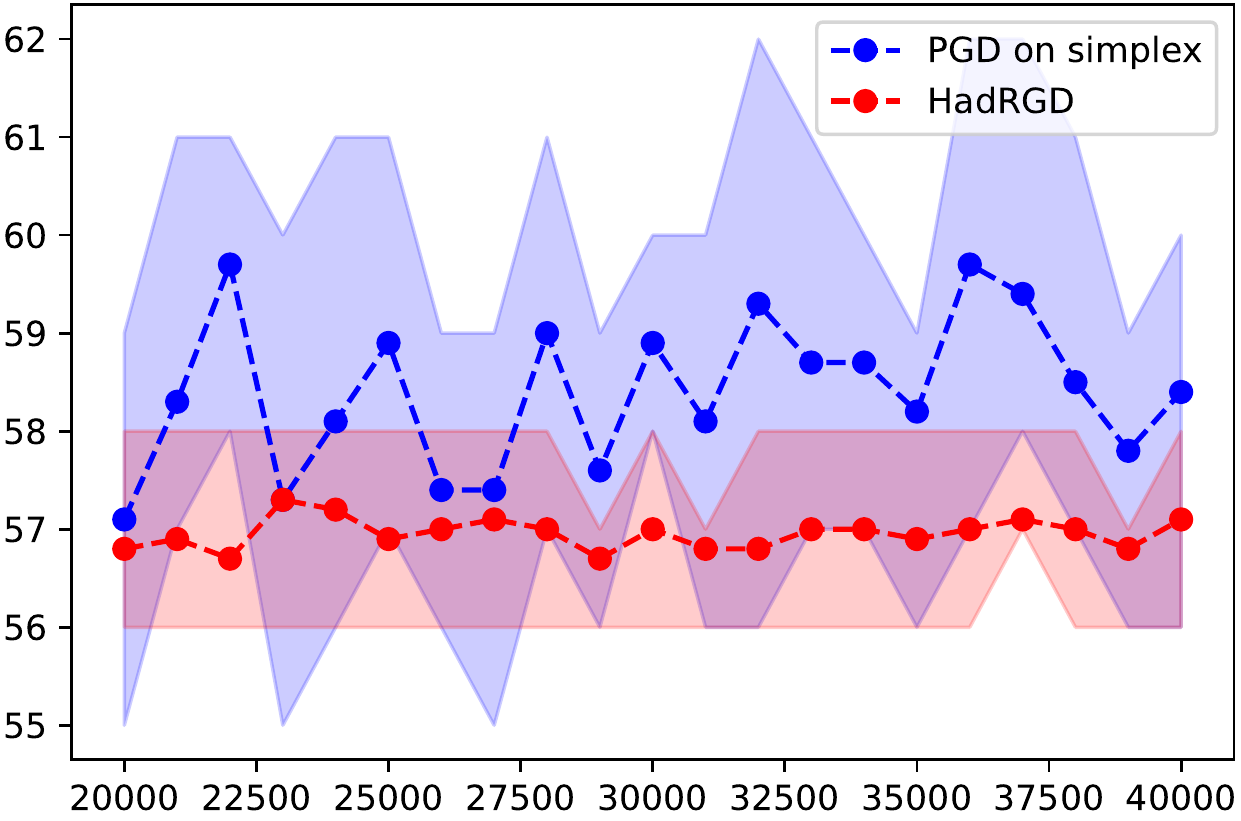}~~~~~
    \includegraphics[width=0.46\linewidth]{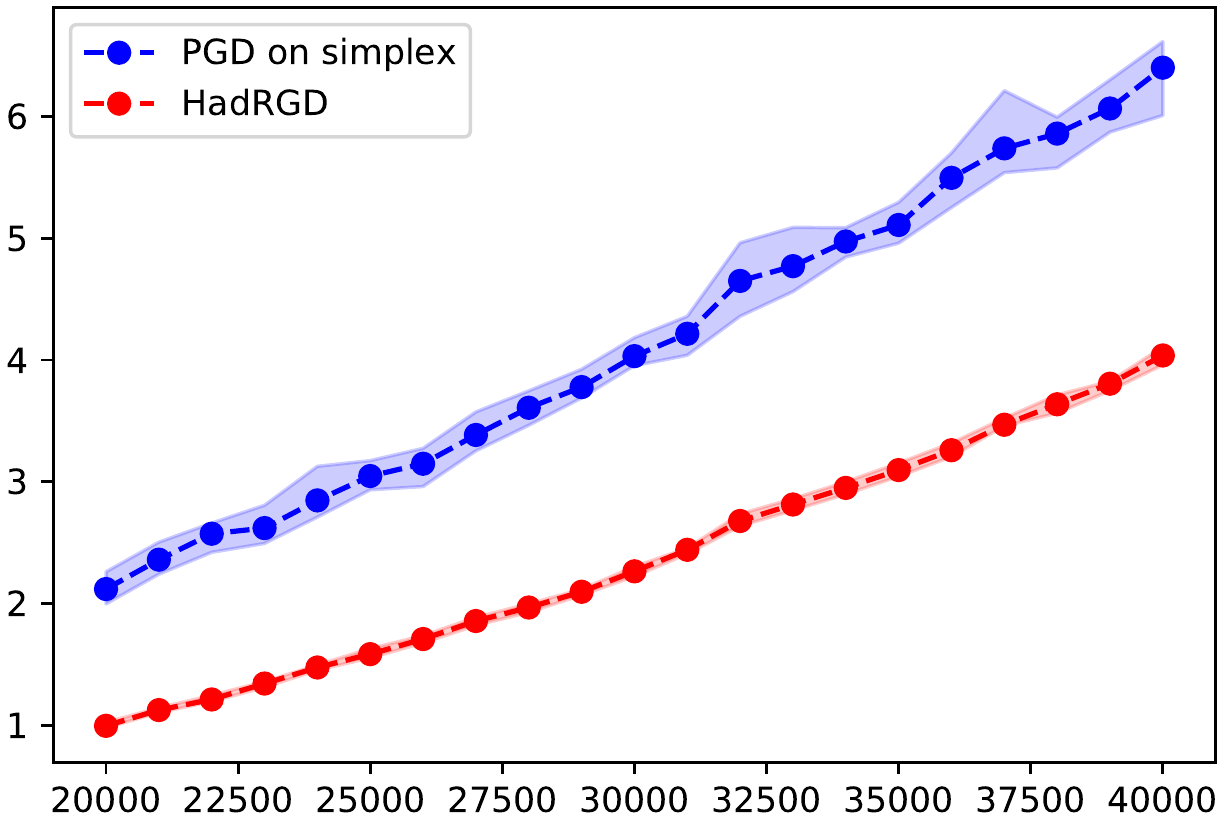}
    \caption{{\bf Left:} Number of iterations vs. $n$. {\bf Right:} Wall-clock time vs. $n$. When the step-size is properly tuned, PRGD requires fewer iterations than PGD on the simplex. Moreover, the per-iteration cost of PRGD is lower, leading to a lower overall run-time.}
    \label{fig:LeastSquares_Interior}
\end{figure*}

Clearly, HadRGD is effective when using an appropriately large step-size, but it is unclear from theoretical grounds how large this step-size can be. Thus, we implement HadRGD in conjunction with a line search algorithm. We consider two approaches.

\paragraph{Armijo-Wolfe} Suppose $v \in T_{z}\mathcal{S}_n$ is a {\em descent direction}\footnote{{\em i.e.} $\langle \grad g(z), v\rangle < 0$} for $g$. We say $\alpha^\star$ satisfies the (Riemannian) Armijo-Wolfe conditions along the geodesic traced out by $\exp_{z}(\alpha v)$ if
\begin{subequations}
\begin{align}
& g\left(\exp_{z}(\alpha^\star v)\right) \leq g\left(z\right) + \rho_1\alpha^\star\langle \grad g(z), v\rangle \label{eq:Armijo}\\
& g^{\prime}\left(\exp_{z}(\alpha^\star v)\right) \geq \rho_2\langle \grad g(z), v\rangle \label{eq:CurvCond}
\end{align}
\end{subequations}
where $g^{\prime}\left(\exp_{z}(\alpha v)\right)$ is shorthand for the derivative of $\alpha \mapsto g\left(\exp_{z}(\alpha v)\right)$. Informally, \eqref{eq:Armijo} guarantees sufficient descent while \eqref{eq:CurvCond} guarantees approximate stationarity (recalling $\langle \grad g(z), v\rangle < 0$). We implement HadRGD, with $\alpha_k$ chosen to satisfy the Armijo-Wolfe conditions along $\exp_{z_k}(\alpha\grad g(z_k))$ via backtracking line search, as HadRGD-AW (Alg.~\ref{alg:HadRGD-AW} in Appendix~\ref{app:Algos}).

\paragraph{Barzilai-Borwein (BB)} The BB step-size\footnote{There is another, closely related BB step-size rule. See \citep{wen2013feasible} for discussion}  rule is
\begin{equation}
    \alpha^{\mathrm{BB}}_k = {\|s_{k-1}\|_2^2}/{\langle s_{k-1},y_{k-1}\rangle} 
\end{equation}
where $s_{k-1} = z_{k} - z_{k-1}$ and $y_{k-1} = \grad g(z_k) - \grad g(z_{k-1})$. HadRGD with the BB step-size need not be monotone, thus we follow \citep{wen2013feasible,zhang2004nonmonotone} and determine $\alpha_k$ via a non-monotone line search starting at $\alpha^{\mathrm{BB}}_k$. That is, we select $z_{k+1} = \exp_{z_k}(\alpha_k\grad g(z_k))$ where $\alpha_k = \delta^h\alpha^{\mathrm{BB}}_k$ where $h$ is the smallest integer satisfying 
\begin{equation*}
    g(\exp_{z_k}(\alpha_k\grad g(z_k))) \leq C_k - \rho_1\alpha_k\|\grad g(z_k)\|^2
\end{equation*}
where $C_k$ is a running average:
\begin{align*}
    C_{k+1} &= \frac{\eta Q_kC_k + g(z_{k+1})}{Q_{k+1}}~\text{ and }~
    Q_{k+1} = \eta Q_k + 1
\end{align*}
We implement this as HadRGD-BB (Alg.~\ref{alg:HadRGD-BB} in Appendix~\ref{app:Algos}). \\

All parameters ($\rho_1,\rho_2,\delta, \eta, \ldots$) are set to the values suggested in \citep{wen2013feasible}. See Appendix~\ref{app:Algos} for further implementation details. 

\subsection{Local Linear Convergence}
We show HadRGD-AW enjoys a locally linear convergence rate, in the case where $x^\star$ is non-degenerate.


\begin{theorem}
Suppose all second-order KKT points of \eqref{C_1} are non-degenerate. Let $\{z^k\}$ be a sequence of points converging to $z^{\star}$ constructed by HadRGD-AW. Then there exists a constants $C,\theta >0$ and an integer $K_0 \geq 0$ such that 
\begin{equation}
    \|z^k - z^\star\|_2 \leq C\theta^k \text{ for } k \geq K_0.
\end{equation}
\label{thm:AW_Linear_Convergence}
\end{theorem}

\begin{proof}
Let $x^\star$ be any second-order KKT point of \eqref{C_1} which is non-degenerate by assumption. By Theorem~\ref{thm:Non-degen_equiv} $z^\star$ is a non-degenerate second-order KKT point of Problem~\eqref{NC_1}, and so is a non-degenerate second-order stationary point of $g$ in the sense of Definition~\ref{def:SOSP}. The claimed convergence rate then follows from \citep[Theorem 4.1]{yang2007globally}.
\end{proof}

\begin{remark}
We observe that HadRGD-AW converges linearly when applied to the underdetermined least-squares problem \eqref{eq:Underdetermined_Least_Squares}, which does not have non-degenerate second-order KKT points, see Appendix~\ref{sec:AdditionalBenchmarking}. Most likely a weakened notion of non-degeneracy such as the restricted strong convexity property \citep{zhang2017restricted} suffices to ensure linear convergence. We leave the exploration of this direction to future work.  
\end{remark}

\section{Extensions}
\label{sec:extensions}

We now extend our framework from $\Delta_n$ to several related geometries, deferring all proofs to Appendix \ref{sec:landscape:proof}. 

\subsection{The unit simplex} Let $\text{\ding{115}}_n$ denote the unit simplex
\begin{equation}
\text{\ding{115}}_n := \left\{x \in \mathbb{R}^n: \ \sum_i [x]_i \leq 1 \text{ and } [x]_i \geq 0 ,\ \forall i \right\} 
\end{equation}
Note that the probability simplex $\Delta_n$ is subset of the boundary of the unit simplex $\text{\ding{115}}_n$. We now consider
\begin{equation}
    \operatorname*{minimize}_{x \in \text{\ding{115}}_n} f(x)
    \tag{$\mathrm{C_2}$}
    \label{C_2}
\end{equation}
Using $x = z\odot z$ we transform \eqref{C_2} to a problem on the unit ball $\mathcal{B}_n^2:=\{x\in\R^n:  \|x\|_2\le1\}$:
\begin{equation}
    \operatorname*{minimize}_{z \in \mathcal{B}_n^2}  f(z\odot z)
    \tag{$\mathrm{NC_2}$}
    \label{NC_2}
\end{equation}
Similar to the probability simplex $\Delta_n$ case,  \eqref{NC_2} has a benign landscape.
\begin{theorem}
\label{thm:H2}
\begin{enumerate}
    \item Suppose that $x^\star$ is a second-order KKT point of \eqref{C_2}. Then, all $z^\star$ satisfying $z^\star\odot z^\star = x^\star$ are second-order KKT points of \eqref{NC_2}.
    
    \item Conversely, suppose $z^\star$ is a second-order KKT point of \eqref{NC_2}. Then, $x^\star = z^\star\odot z^\star$ is a second-order KKT point of \eqref{C_2}.
\end{enumerate}
\end{theorem}

\subsection{The weighted probability simplex} For any $a\in\mathbb{R}^n$ with $a>0$ (entrywise), consider the weighted probability simplex:
\begin{equation}
    \Delta_{n}^a := \left\{z \in \mathbb{R}^{n}:  \sum_{i=1}^{n} [a]_i [z]_i = 1  \text{ and } [z]_i\ge 0,\ \forall i\right\}
\end{equation}
Let us consider the following  optimization problem:
\begin{equation}
   \operatorname*{minimize}_{x \in \Delta_{n}^a} f(x)
    \tag{$\mathrm{C_3}$}
    \label{C_3}
\end{equation}
Using once more $x = z\odot z$, we obtain:
\begin{equation}
   \operatorname*{minimize}_{z \in \mathcal{B}_n^a}  f(z\odot z)
    \tag{$\mathrm{NC_3}$}
    \label{NC_3}
\end{equation}
Here, the unit $a$-weighted $\ell_2$ norm ball is defined as
\[
\mathcal{B}_n^a:=\left\{x\in\mathbb{R}^n: \|x\|^2_{\mathrm{diag}(a)}=1\right\}
\]
where  
\[\|x\|^2_{\mathrm{diag}(a)}:=\sum_i [a]_i [x]_i^2\]
We claim \eqref{NC_3} enjoys a benign  landscape.
\begin{theorem}
\begin{enumerate}
    \item Suppose that $x^\star$ is a second-order KKT point of \eqref{C_3}. Then, all $z^\star$ satisfying $z^\star\odot z^\star = x^\star$ are second-order KKT points of \eqref{NC_3}.
    
    \item Conversely, suppose $z^\star$ is a second-order KKT point of \eqref{NC_3}. Then, $x^\star = z^\star\odot z^\star$ is a second-order KKT point of \eqref{C_3}.
\end{enumerate}
\label{thm:H3}
\end{theorem}

\subsection{The \texorpdfstring{$\ell_1$}{} norm ball} 
Consider the minimization:
\begin{equation}
   \operatorname*{minimize}_{x\in\mathcal{B}^{1}_n}f(x)
    \label{C_4}
    \tag{$\mathrm{C_4}$}
\end{equation}
where the unit $\ell_1$ ball is defined as
\[\mathcal{B}_n^1:=\{x\in\R^n:  \|x\|_1\le1\}\] 
Since $x$ may have negative entries, we use the double Hadamard parametrization $x = z_u\odot z_u - z_v\odot z_v$ and transform \eqref{C_4} to:
\begin{equation}
\operatorname*{minimize}_{(z_u,z_v)\in \mathcal{B}^{2}_{2n}}f(z_u\odot z_u - z_v\odot z_v)
    \tag{$\mathrm{NC_4}$}
    \label{NC_4}
\end{equation}
Due to the nonsmoothness of $\|x\|_1$ (not twice continuously differentiable) in the Lagrangian function of \eqref{C_4}, we consider a simpler situation here.  
\begin{theorem}
Assume $f$ is convex. 
\begin{enumerate}
    \item Suppose that $x^\star$ is a second-order KKT point of \eqref{C_4}. Then, all $(z_u^\star,z^\star_v)$ satisfying ${z_u^\star\odot z_u^\star - z_v^\star\odot z_v^\star} = x^\star$ are second-order KKT points of \eqref{NC_4}.
    
    \item Conversely, suppose $(z_u^\star,z^\star_v)$ is a second-order KKT point of \eqref{NC_4}. Then, $x^\star = {z_u^\star\odot z_u^\star - z_v^\star\odot z_v^\star}$ is a second-order KKT point of \eqref{C_4}.
\end{enumerate}
\label{thm:H4}
\end{theorem}


    




\begin{center}
\begin{figure*}[!h]
    \centering
    \includegraphics[width=0.33\linewidth]{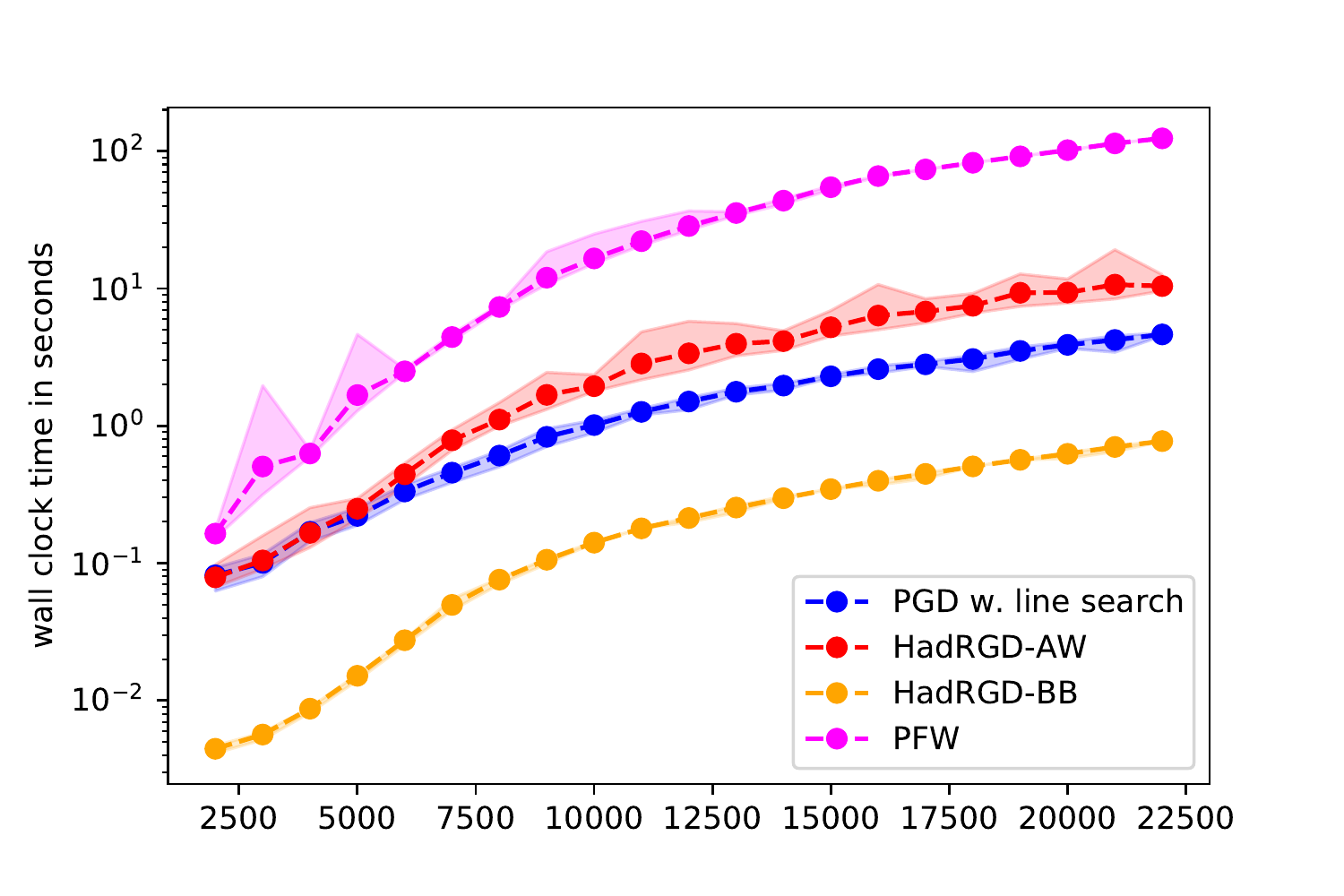}
     \includegraphics[width=0.33\linewidth]{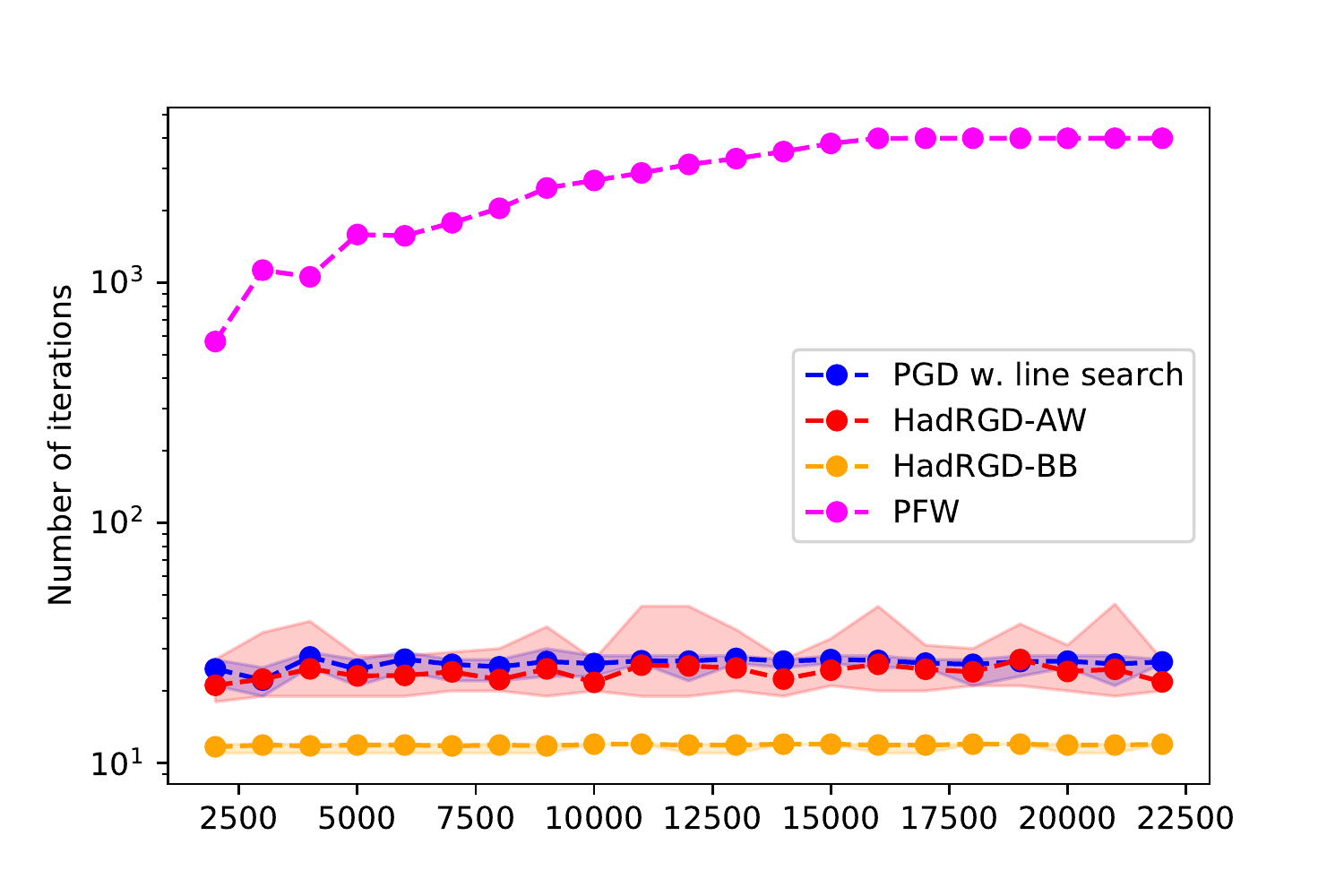}
    \includegraphics[width=0.32\linewidth]{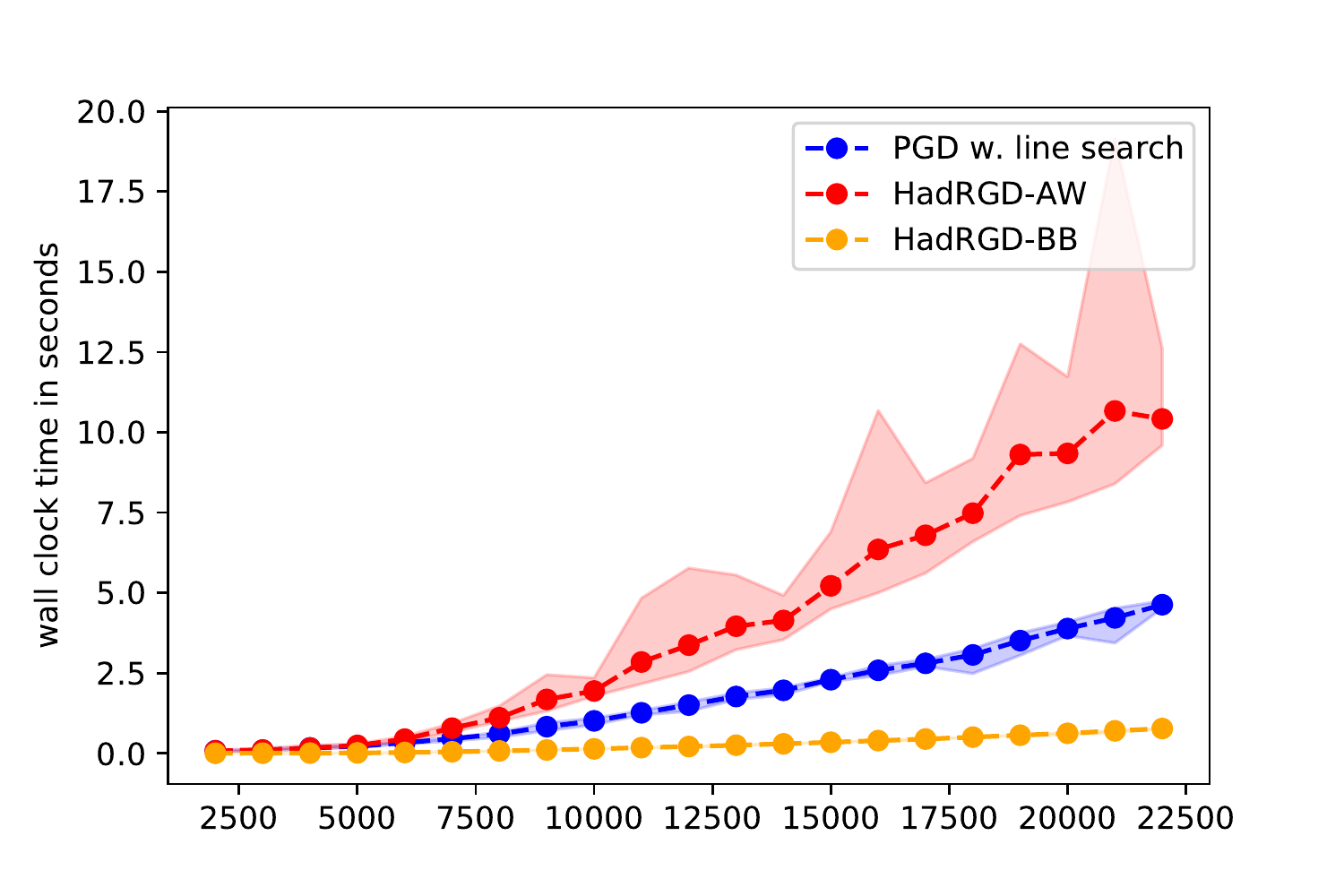}
    \caption{Solving underdetermined least squares with $x_{\mathrm{true}}\in \mathrm{int}(\Delta_n)$ and a target solution accuracy of $10^{-8}$. {\bf Left:} Time required. {\bf Center:} Number of iterations required. {\bf Right:}  Time required, for the three fastest algorithms, not in log scale. All results are averaged over ten trials and the shading denotes the min-max range.}
    \label{fig:LeastSquares_Interior}
\end{figure*} 

~

\begin{figure*}[!h]
    \centering
    \includegraphics[width=0.33\linewidth]{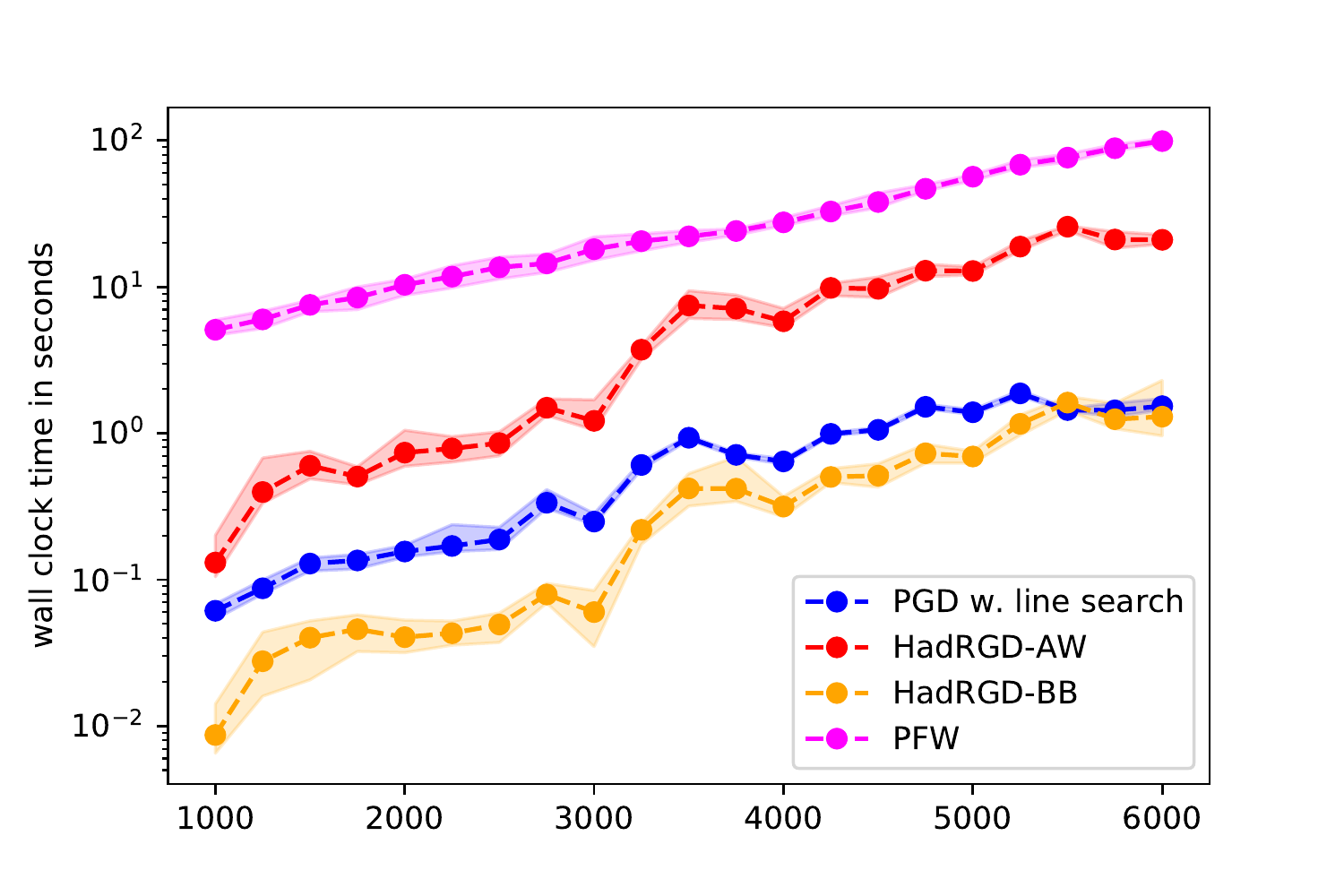}
     \includegraphics[width=0.33\linewidth]{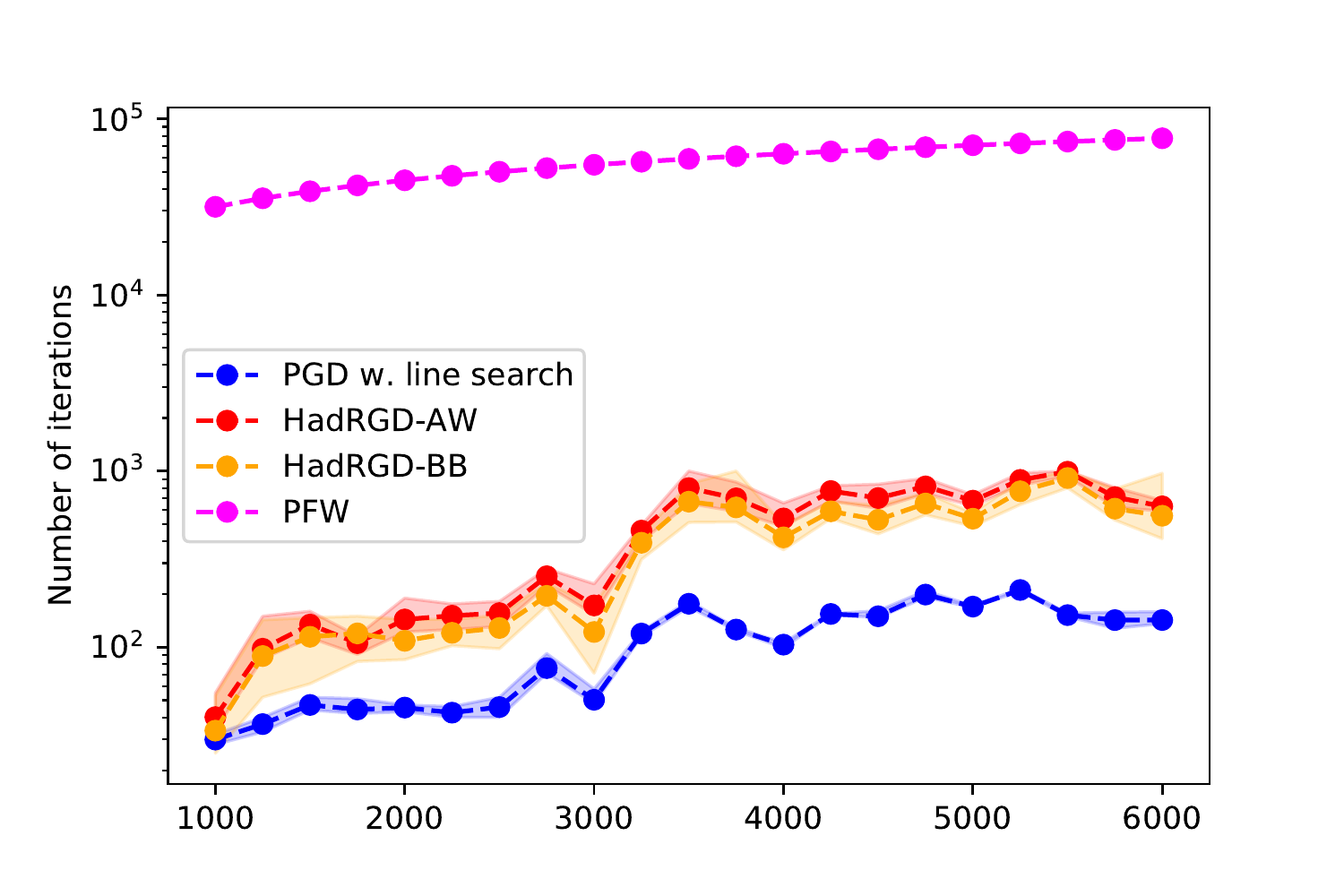}
    \includegraphics[width=0.32\linewidth]{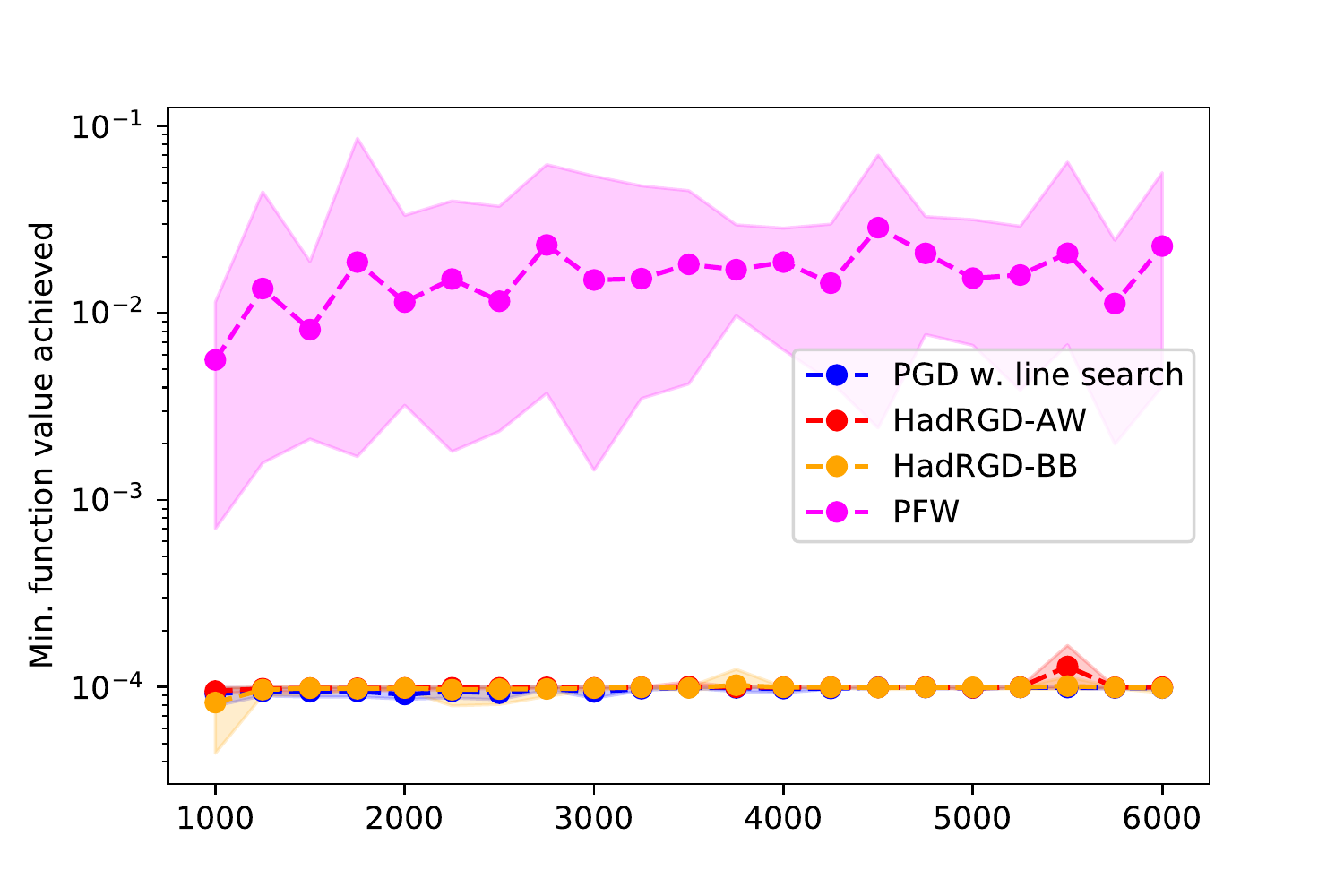}
    \caption{Solving underdetermined least squares with $x_{\mathrm{true}}$ on the boundary of $\Delta_n$ and a target solution accuracy of $10^{-4}$. {\bf Left:} Time required. {\bf Center:} Number of iterations required. {\bf Right:} Final solution accuracy; PFW struggles to reach the target accuracy. All results are averaged over ten trials and the shading denotes the min-max range.}
    \label{fig:LeastSquares_Boundary}
\end{figure*}
\end{center}

\section{Experiments}
\label{sec:Experiments}
We consider again the underdetermined simplex-constrained least squares problem \eqref{eq:Underdetermined_Least_Squares}. We focus on this test problem as it is a key component of algorithms for applications as diverse as portfolio optimization \citep{bomze2002regularity}, archetypal analysis \citep{bauckhage2015archetypal} and hyperspectral unmixing \citep{condat2016fast}. As before $A\in \bbR^{m\times n}$ with $m = 0.1n$ and $b = Ax_{\mathrm{true}}$, but we consider two qualitatively different methods for selecting $x_{\mathrm{true}}$: (i) $x_{\mathrm{true}}$ is sampled uniformly at random from the interior of $\Delta_n$ and (ii) $x_{\mathrm{true}}$ is the projection of a Gaussian random vector to $\Delta_n$ (and hence lies on the boundary of $\Delta_n$). We considered the following benchmark algorithms: Entropic Mirror Descent Algorithm (EMDA) \citep{ben2001ordered,beck2003mirror}; Pairwise Frank-Wolfe (PFW) \citep{pedregosa2020linearly}, both with and without linesearch; and Projected Gradient Descent (PGD) on $\Delta_n$, with and without linesearch. See Appendix~\ref{sec:ExperimentalSettings} for more details and \url{https://github.com/DanielMckenzie/HadRGD} for code. We found EMDA and PFW (without line search) to be non-competitive in both cases (see Appendix~\ref{sec:AdditionalExperiments}), hence here we only present results for PGD (with line search), HadGrad-AW, HadGrad-BB and PFW (with line search). The results for case (i) are presented in Figure~\ref{fig:LeastSquares_Interior} while the results for case (ii) are presented in Figure~\ref{fig:LeastSquares_Boundary}. We record, as a function of $n$, the wall-clock time required to find an $\varepsilon$-solution ($\varepsilon=10^{-8}$ for case (i) and $\varepsilon=10^{-4}$ for case (ii)) or until the maximum number of iterations is reached (1000 for PGD, HadGrad-AW, HadGrad-BB and $1000\sqrt{n}$ for PFW). \\

From Figure~\ref{fig:LeastSquares_Interior} it is clear that HadGrad-BB is the fastest algorithm for problem~\eqref{eq:Underdetermined_Least_Squares}, in the case where there are solutions in $\mathrm{int}(\Delta_n)$, converging an order of magnitude than the others, both in terms of number of iterations and wall-clock time. The situation is less clear for solutions on the boundary of $\Delta_n$ (see Figure~\ref{fig:LeastSquares_Boundary}). We note that a significant amount of the run-time of HadGrad-AW is spent on computing points along the geodesic $\exp_z(\alpha v)$. This can be ameliorated by using a {\em retraction} within the line search, instead of a geodesic \citep{boumal2020introduction}. We leave this for future work.

\section{Conclusion}
In this paper we presented a new framework for transforming a non-smooth constraint set to the unit sphere or ball. We showed that the landscape of the transformed problem is benign if the landscape of the original problem is. Using techniques from Riemannian optimization, we constructed a new algorithm guaranteed to converge to second-order KKT points of the original problem, and proved a rate of convergence for a certain special case. We explored augmenting our new algorithm with various line-search sub-routines. Empirically, we showed that our algorithm could be significantly faster than existing algorithms, particularly when the solution to the original problem lies in the interior of the simplex. Future work could combine our Hadamard reparametrization trick with proximal or sub-gradient methods for non-smooth $f$, or SGD-style methods for finite sum objectives: $f(x) = \sum_{i=1}^nf_i(x)$.


\bibliography{References}


\appendix

\onecolumn

\section{Proof of Hadamard Calculus}

\begin{description}
    \item[{($\mathrm{H1}$)}]  $1_n\odot z = z$ where $1_n$ is the all-one vector in $\R^n$
     \begin{proof}
    This follows from the definition of Hadamard product.
    \end{proof}
    
    \item[{($\mathrm{H2}$)}] If $d\odot z\odot z = 0$ then $d\odot z = 0$
    
  \begin{proof}
   We can multiply $d$ on both sides of left-hand-side equation: $d\odot d\odot z\odot z=0$, which implies $(d\odot z)\odot (d\odot z)=0$. So, Therefore, $d\odot z=0$.
  \end{proof}

    \item[{($\mathrm{H3}$)}] 
        $ \mathrm{diag}(z)d = z\odot d $ and
        $ d^{\top}\mathrm{diag}(z)d = \langle d, z\odot d \rangle= \langle z,d\odot d \rangle $
         
  \begin{proof}
    The first line is because $[\diag(z) d]_k = [z]_k [d]_k =[z\odot d]_k$. The second line follows by multiplying $d^\top$ on both sides and $\langle d, z\odot d \rangle=\sum_k [d]_k [z]_k [d]_k = \sum_k [z]_k [d]_k^2 =\langle z,d\odot d \rangle. $
  \end{proof}

    \item[{($\mathrm{H4}$)}] $\|d\odot z\|_2 \leq \|d\|_2\|z\|_{\infty}$
    
  \begin{proof}
    Note that 
    \[\|d\odot z\|_2=\sqrt{\sum_k [d]_k^2 [z]_k^2}\le \sqrt{\sum_k  [d]_k^2 z_{\max}^2} = z_{\max} \sqrt{\sum_k  [d]_k^2}=\|d\|_2\|z\|_{\infty}\] 
    where $z_{\max}^2:=\max_k [z]_k^2$.
    \end{proof}

    \item[{($\mathrm{H5}$)}] $\nabla g(z) = 2 \nabla f(x) \odot z$  and $\nabla^2 g(z) = 2\mathrm{diag}(\nabla f(x))+4 \mathrm{diag}(z) \nabla^2 f(x) \mathrm{diag}(z)$.
    
  \begin{proof}
    One one hand, the Taylor expansion of $g$ around $z$ for some arbitrary small direction $d$ is given by 
   \begin{align}
   g(z+d) = g(z) + d^\top \nabla g(z) +  \frac{1}{2} d^\top \nabla^2 g(z) d + o(\|d\|^2_2)
   \label{H5:Proof:1}
   \end{align}
   On the other hand, we have 
   \begin{align}
       g(z+d)&=f((z+d)\odot (z+d))=f(z\odot z + 2 z\odot d +  d\odot d) 
       \nonumber\\
       &= f(z\odot z) + (2z\odot d + d\odot d)^\top\nabla f(z\odot z) + \frac{1}{2} (2 z\odot d)^\top \nabla^2 f(z\odot z) (2z\odot d)+o(\|d\|^2) 
          \nonumber    \\
       &= g(z) + 2d^\top(\nabla f(x)\odot z)+  (d\odot d)^\top\nabla f(x) + \frac{4}{2}  d^\top [\diag(z)\nabla^2 f(x)\diag(z)] d+o(\|d\|^2)
       \nonumber\\
      &= g(z) + 2d^\top(\nabla f(x)\odot z)+ 
      d^\top \diag(\nabla f(x)) d +  2 d^\top [\diag(z)\nabla^2 f(x)\diag(z)] d+o(\|d\|^2)
      \label{H5:Proof:2}
   \end{align}
   Combining and rearranging \eqref{H5:Proof:1} and \eqref{H5:Proof:2}, we complete the proof.
   
   \end{proof}
\end{description}

\newpage 

\section{Proof of Lemma~\ref{lemma:Lipschitz}}
\label{app:RiemGeom}
\noindent {\bf Lemma~\ref{lemma:Lipschitz}.} \textit{Suppose $f(x)$ is $L$-Lipschitz differentiable. Then $g$ is $\tilde{L}$-Lipschitz differentiable with $\tilde{L} = 4L + 2M$ where $M = \max_{x\in \Delta_n}\|\nabla f(x)\|_2$.} \\

\begin{proof} 
Recalling $\nabla_zg(z) = 2\nabla_xf(z\odot z)\odot z$ we compute
\begin{align}
\| \nabla g(z_1) - \nabla g(z_2)\|_2 & = 2\|\nabla_x f(z_1\odot z_1)\odot z_1 - \nabla_x f(z_2\odot z_2)\odot z_2 \|_2 \nonumber \\
    & = 2\|\nabla_x f(z_1\odot z_1)\odot z_1  - \nabla_x f(z_2\odot z_2)\odot z_1 + \nabla_x f(z_2\odot z_2)\odot z_1 - \nabla_x f(z_2\odot z_2)\odot z_2 \|_2 \nonumber\\
    & \leq 2\|\left(\nabla_x f(z_1\odot z_1) - \nabla_x f(z_2\odot z_2)\right)\odot z_1\|_2 + 2\|\nabla_x f(z_2\odot z_2)\odot\left(z_1 - z_2\right)\|_2 \nonumber \\
    & \leq 2\|\nabla_x f(z_1\odot z_1) - \nabla_x f(z_2\odot z_2)\|_2 \|z_1\|_{\infty} + 2\|\nabla_x f(z_2\odot z_2)\|_{\infty} \|z_1 - z_2\|_2 \nonumber \\ 
    & \leq 2 \left(L \|z_1\odot z_1 - z_2 \odot z_2\|_2\right)\left(1\right) + 2M \|z_1 - z_2\|_2 \nonumber \\
    & \leq 2L\left(2\|z_1 - z_2\|_2 \right)+ 2M \|z_1 - z_2\|_2 \nonumber \\
    & = \left(4L + 2M\right)\|z_1 - z_2\|_2 \nonumber
\end{align}
where the second inequality follows from \ref{H4} while the last inequality follows from Lemma~\ref{lem:LittleLemma}.
\end{proof}

\begin{lemma}
If $z_1,z_2 \in \mathcal{S}_n$ then $\|z_1\odot z_1 - z_2\odot z_2\|_2 \leq 2\|z_1 - z_2\|_2$.
\label{lem:LittleLemma}
\end{lemma}

\begin{proof}
\begin{align*}
    \|z_1\odot z_1 - z_2\odot z_2\|_2 & = \|z_1\odot z_1-z_1\odot z_2 + z_1\odot z_2 - z_2\odot z_2\|_2 \\
    & \leq \|z_1\odot (z_1 - z_2)\|_2 + \|(z_1 - z_2)\odot z_2\|_2 \\
    & \leq \|z_1\|_{\infty}\|z_1 - z_2\|_2 + \|z_1 - z_2\|_2\|z_2\|_{\infty} \\
    & \leq (1) \|z_1 - z_2\|_2 + \|z_1 - z_2\|_2 (1)  = 2\|z_1 - z_2\|_2
\end{align*}
where the second inequality follows from \ref{H4}.
\end{proof}

\section{Algorithms}
\label{app:Algos}
Recall $\alpha^\star$ satisfies the (Riemannian) Armijo-Wolfe conditions along the geodesic traced out by $\exp_{z}(\alpha v)$ if
\begin{subequations}
\begin{align}
& g\left(\exp_{z}(\alpha^\star v)\right) \leq g\left(z\right) + \rho_1\alpha^\star\langle \grad g(z), v\rangle \label{eq:Armijo}\\
& g^{\prime}\left(\exp_{z}(\alpha^\star v)\right) \geq \rho_2\langle \grad g(z), v\rangle \label{eq:CurvCond}
\end{align}
\end{subequations}
When $v = -\operatorname{grad}g(z)$ this simplifies to:
\begin{subequations}
\begin{align}
& g\left(\exp_{z}(\alpha^\star v)\right) \leq g\left(z\right) - \rho_1\alpha^\star\|\operatorname{grad}g(z)\|_2^2 \label{eq:Armijo}\\
& g^{\prime}\left(\exp_{z}(\alpha^\star v)\right) \geq \rho_2\|\operatorname{grad}g(z)\|_2^2 \label{eq:CurvCond}
\end{align}
\end{subequations}
We use this version of the Armijo-Wolfe conditions in Algorithm~\ref{alg:HadRGD-AW}, with the notation $\nabla_k = \operatorname{grad}g(z_k)$.
\begin{algorithm}
\caption{HadPRGD for \eqref{C_1}}
\label{alg:HadPRGD}
\textbf{Input}: $x_0 \in \Delta_n$: initial point, $\alpha$: step size, $K$: number of iterations, $f$: original objective function
\begin{algorithmic}[1] 
\STATE $z_0 = \sqrt{x_0}$ \COMMENT{(Defined componentwise)}
\STATE $g(z) := f(z\odot z)$
\FOR{k=1,\ldots, K}
    \IF{$\|\operatorname{grad} g(z_k)\|_2 > \varepsilon$}
        \STATE $z_{k+1} = \operatorname{exp}_{x_k}(-\alpha\operatorname{grad}f(x_k))$
    \ELSE 
        \STATE $\xi \sim \mathrm{Uniform}(B_{z_k,r}(0))$
        \STATE $s_0 = \eta\xi$
        \STATE $z_{k+\mathcal{T}} \gets \mathrm{TangentSpaceSteps}(z_k, s_0, \eta, b, \mathcal{T})$ \COMMENT{(See Algorithm~\ref{alg:TangentSpaceSteps})}
    \ENDIF
    
\ENDFOR
\end{algorithmic}
\textbf{Return} $x_{K} = z_{K}\odot z_{K}$
\end{algorithm}

\begin{algorithm}
\caption{TangentSpaceSteps}
\label{alg:TangentSpaceSteps}
\textbf{Input}: $z, s_0, \alpha, \eta, b, \mathcal{T}$.
\begin{algorithmic}[1] 
\FOR{$j=1,\ldots, \mathcal{T}$}
    \STATE $s_{j+1}= s_j - \eta\grad f_z(s_j)$
    \IF{$\|s_{j+1}\|_2 \geq b$}
        \STATE $s_{\mathcal{T}} = s_{j} - \alpha\eta\nabla \hat{f}_z(s_j)$
        \STATE Break.
    \ENDIF
\ENDFOR
    \STATE Return
    $\operatorname{Proj}_z(s_{\mathcal{T}})$
\end{algorithmic}
\textbf{Return} $x_{K} = z_{K}\odot z_{K}$
\end{algorithm}

\begin{algorithm}
\caption{HadRGD-AW for \eqref{C_1}}
\label{alg:HadRGD-AW}
\textbf{Input}: $x_0 \in \Delta_n$: initial point, $\alpha_{\mathrm{def}}$: default step size, $\beta$: decay factor, $\rho_1$: Armijo condition tolerance, $\rho_2$: Wolfe condition tolerance, $K$: number of iterations, $f$: original objective function
\begin{algorithmic}[1] 
\STATE $z_0 = \sqrt{x_0}$ \COMMENT{(Defined componentwise)}
\STATE $g(z) := f(z\odot z)$
\FOR{$k=1,\ldots,K$}
    \STATE $\nabla_k \gets \operatorname{grad}g(z_k)$
    \STATE $m=0$
    \STATE {\tt ArmijoFlag}=False
    \STATE {\tt WolfeFlag}=False
    \WHILE{{\tt ArmijoFlag}=False and {\tt WolfeFlag}=False and $m \leq 25$}
        \STATE $\alpha \gets \alpha_{\mathrm{def}}\beta^m$
        \STATE $m \gets m+1$
        \IF{$g\left(\exp_{z_k}(\alpha \nabla_k)\right) \leq g\left(z\right) - \rho_1\alpha\|\nabla_k\|_2^2$}
            \STATE {\tt ArmijoFlag}=True
        \ENDIF
        \IF{$g^{\prime}\left(\exp_{z_k}(\alpha^\star v)\right) \geq -\rho_2\|\nabla_k\|_2^2$}
            \STATE {\tt WolfeFlag}=True
        \ENDIF
    \ENDWHILE
    \STATE $z_{k+1} \gets \operatorname{exp}_{z_k}\left(-\alpha\nabla_k\right)$
\ENDFOR
\end{algorithmic}
\textbf{Return} $x_{K} = z_{K}\odot z_{K}$
\end{algorithm}

\begin{algorithm}
\caption{HadRGD-BB for \eqref{C_1}}
\label{alg:HadRGD-BB}
\textbf{Input}: $x_0 \in \Delta_n$: initial point, $\alpha_{\mathrm{def}}$: default step size, $\delta$: decay factor, $\eta$: moving average factor, $\rho_1$: tolerance factor, $K$: number of iterations, $f$: original objective function
\begin{algorithmic}[1] 
\STATE $z_0 = \sqrt{x_0}$ \COMMENT{(Defined componentwise)}
\STATE $g(z) := f(z\odot z)$
\STATE $\alpha \gets \alpha_{\mathrm{def}}$
\FOR{$k=1,\ldots,K$}
    \STATE $\nabla_k \gets \operatorname{grad} g(z_k)$
    \WHILE{$g(\operatorname{exp}_{z_k}(-\alpha\nabla_k) \geq C_k - \rho_1\alpha\|\nabla_k\|_2^2$}
        \STATE $\alpha \gets \alpha\delta$
    \ENDWHILE
    \STATE $z_{k+1} \gets \operatorname{exp}_{z_k}(-\alpha\nabla_k)$
    \STATE $Q_{k+1} \gets \eta Q_k + 1$
    \STATE $C_{k+1} \gets \left(\eta Q_kC_k + g(z_{k+1})\right)/Q_{k+1}$
    \STATE $s_{k+1} \gets z_{k+1} - z_k$
    \STATE $y_{k+1} \gets \operatorname{grad}g(z_{k+1}) - \operatorname{grad}g(z_k)$
    \STATE $\alpha^{\mathrm{BB}}_{k+1} \gets \frac{\|s_{k+1}\|_2^2}{\left|\langle s_{k+1},y_{k+1}\rangle\right|}$
    \STATE $\alpha \gets \max\left\{\min\left\{\alpha^{\mathrm{BB}}_{k+1},30 \right\}10^{-10}\right\}$
\ENDFOR
\end{algorithmic}
\textbf{Return} $x_{K} = z_{K}\odot z_{K}$
\end{algorithm}

\begin{algorithm}
\caption{PGD with line search for \eqref{C_1}}
\label{alg:PGDL}
\textbf{Input}: $x_0 \in \Delta_n$: initial point, $P_{\Delta_n}(\cdot)$: A callable projection algorithm, $s$: step size, $\beta$: decay factor, $\rho_1$: Armijo condition tolerance $K$: number of iterations, $f$: original objective function
\begin{algorithmic}[1] 
\FOR{$k=1,\ldots,K$}
    \STATE $\nabla_k \gets \nabla f(x_k)$
    \STATE $\bar{x} \gets P_{\Delta_n}\left(x_k - s\nabla_k\right)$
    \STATE $m=0$
    \STATE {\tt ArmijoFlag}= False
    \WHILE{{\tt ArmijoFlag}= False and $m \leq 25$}
        \STATE $\alpha \gets \beta^m$
        \STATE $x_{\mathrm{new}} \gets x_k + \alpha\left(\bar{x} -x_k\right)$
        \IF{$f(x_k) - f(x_{\mathrm{new}}) \geq -\rho_1\langle \nabla_k,\bar{x} - x_k\rangle$}
            \STATE {\tt ArmijoFlag}= True
        \ENDIF
    \ENDWHILE
    \STATE $x_{k+1} \gets x_{\mathrm{new}}$
\ENDFOR
\end{algorithmic}
\textbf{Return} $x_{K} = z_{K}\odot z_{K}$
\end{algorithm}
\newpage

\section{Proofs of Landscape Analysis}
\label{sec:landscape:proof}

\subsection{The Unit Simplex --- Proof of Theorem \ref{thm:H2}}
Recall the original problem is
\begin{equation}
    x^{\star} \in \argmin_{x \in \text{\ding{115}}_n} f(x)
    \tag{$\mathrm{C_2}$}
\end{equation}
and the Hadamard parameterized problem is
\begin{equation}
    z^{\star} \in \argmin_{z \in \mathcal{B}_n^2}  f(z\odot z)
    \tag{$\mathrm{NC_2}$}
\end{equation}

\paragraph{First-order KKT conditions of Problem~\eqref{C_2}} 
The Lagrangian of Problem~\eqref{C_2} is 
\begin{align}
     \mathcal{L}_C(x,\lambda,\beta) = f(x) - \lambda (1_n^\top x-1) +\beta^\top x 
\end{align}
The first-order conditions  of Problem~\eqref{C_2} are given by: there exist $\lambda^\star\in\mathbb{R}$ and $\beta^\star\in\mathbb{R}^n$ such that 
\begin{subequations}
\begin{align}
\nabla f(x^\star) &= \lambda^\star 1 + \beta^\star
 \label{KKT:C2:1}
 \\
 x^\star& \ge 0
  \label{KKT:C2:2}
 \\
 \beta^\star &\ge 0
  \label{KKT:C2:3}
 \\
 x^\star\odot \beta^\star &=0
  \label{KKT:C2:4}
 \\
 1_n^\top x^\star &\le1
  \label{KKT:C2:5}
 \\
 \lambda ^\star &\le 0
  \label{KKT:C2:6}
 \\
 (1_n^\top x^\star-1)\lambda^\star & = 0
  \label{KKT:C2:7}
\end{align}
 \label{KKT:C2}
\end{subequations}

\paragraph{Second-order KKT conditions of Problem~\eqref{C_2}} 
As 
\[
\nabla^2 f(x) =\nabla^2_x \mathcal{L}_C(x,\lambda,\beta)
\]
the second-order KKT conditions of Problem \eqref{C_2} are given by \eqref{KKT:C2} and 
\begin{equation}
    u^\top\nabla^2f(x^\star)u \geq 0 
    \label{eq:KKT:C2:2}
\end{equation}
for all $u$ such that 
\[
\begin{cases}
1_n^\top u = 0 &\text{if }1_n^\top x^\star=1\\
[u]_k = 0  &\text{if }[x^\star]_k = 0
\end{cases}
\]

\paragraph{First-order KKT conditions of Problem~\eqref{NC_2}}
The Lagrangian of Problem~\eqref{NC_2} is
\begin{align}
    \mathcal{L}_N(z,\lambda_N) = f(z\odot z) - \lambda_N (\|z\|_2^2-1)
\end{align}
Then the first-order KKT conditions of Problem~\eqref{NC_2} are given by that there exist $\lambda_N^\star\in\mathbb{R}$ such that 
\begin{subequations}
\begin{align}
\nabla f(z^\star\odot z^\star)\odot z^\star &= \lambda_N^\star z^\star
\label{KKT:NC2:1}
 \\
\|z^\star\|^2 &\le1
\label{KKT:NC2:2}
\\
\lambda_N^\star &\le 0
\label{KKT:NC2:3}
\\
(\|z^\star\|^2-1))\lambda_N^\star &=0
\label{KKT:NC2:4}
\end{align}
\label{KKT:NC2}
\end{subequations}

\paragraph{Second-order KKT conditions of Problem~\eqref{NC_2}}
The second-order KKT conditions of Problem~\eqref{NC_2} are given by \eqref{KKT:NC2} and that there exists $\lambda_N^\star\in\mathbb{R}$ such that 
\begin{equation}
\begin{aligned}
d^\top\left[\nabla^2 \mathcal{L}_N(z^\star,\lambda^\star_N)\right] d &=
d^\top\left[ 2\mathrm{diag}(\nabla f(z^\star\odot z^\star)) + 4 \mathrm{diag}(z^\star) \nabla^2 f(z^\star\odot z^\star) \mathrm{diag}(z^\star) - 2\lambda_N^\star I_n \right] d 
\nonumber\\
& =
 4(z^\star\odot d)^\top \nabla^2 f(z^\star\odot z^\star)(z^\star\odot d)+2\langle\nabla f(z^\star\odot z^\star), d\odot d\rangle - 2\lambda_N^\star \|d\|^2 \ge 0  
\end{aligned}
\label{KKT2:NC2}
\end{equation}
for all $d$ such that
\[
\begin{cases}
d\in\R^n & \text{if }\|z^\star\|<1\\
d\bot z^\star & \text{if }\|z^\star\|=1
\end{cases}
\]


\begin{proof}[\bf Proof of Theorem \ref{thm:H2}] ~

\begin{enumerate} 
   \item Suppose $x^\star$ is a second-order KKT point of \eqref{C_2}.
   Take any $z^\star$ such that $z^\star\odot z^\star = x^\star$.
   Firstly, by multiplying $z^\star$ on both sides of  \eqref{KKT:C2:1} and choosing $\lambda_N^\star=\lambda^\star$, we recover  the first-order KKT condition \eqref{KKT:NC2:1}. Secondly, by plugging  $x^\star=z^\star\odot z^\star$ and $1_n^\top x^\star=\|z^\star\|^2$ to  \eqref{KKT:C2:5}--\eqref{KKT:C2:7} and choosing $\lambda_N^\star=\lambda^\star$, we recover the first-order KKT conditions \eqref{KKT:NC2:2}--\eqref{KKT:NC2:4}.
   
   Hence, it remains to show the second-order KKT conditions \eqref{KKT2:NC2} with $\lambda_N^\star=\lambda^\star$. 
    Since $x^\star$ satisfies the first-order KKT conditions \eqref{KKT:C2} and  $x^\star=z^\star\odot z^\star$, we get
    \begin{align}
       & \nabla f(x^\star) = \lambda^\star 1_n + \beta^\star ,\beta\ge 0
      \nonumber \\
       \iff & \mathrm{diag}(\nabla f(z^\star\odot z^\star)) - \lambda^\star I_n = \mathrm{diag}(\beta^\star)\succeq 0
       \nonumber\\
      \iff & \langle \nabla f(z^\star\odot z^\star), d\odot d\rangle - \lambda^\star \|d\|^2 \ge 0,\ \forall d
      \label{nabla f - lam}
    \end{align}
    Using \eqref{nabla f - lam}, we have
    \begin{align}
    d^\top[\nabla^2 \mathcal{L}_N(z^\star,\lambda^\star)]d
     &= 4(z^\star\odot d)^\top \nabla^2 f(z^\star\odot z^\star)(z^\star\odot d)+2\langle\nabla f(z^\star\odot z^\star), d\odot d\rangle - 2\lambda^\star \|d\|^2
     \nonumber\\
     &\ge 4(z^\star\odot d)^\top \nabla^2 f(z^\star\odot z^\star)(z^\star\odot d)
    \nonumber\\
     & \ge 0, \begin{cases}\text{for all }
d\in\R^n & \text{if }\|z^\star\|<1\\
\text{for all }d\bot z^\star & \text{if }\|z^\star\|=1
\end{cases}
    \label{H2:PSD}
    \end{align}
    Then, as long as \eqref{H2:PSD} holds, $z^\star$ is a second-order KKT point of Problem~\eqref{NC_2}.    
    
\paragraph{Show the first line of \eqref{H2:PSD}} First, if $\|z^\star\|<1$, we have $1_n^\top (z^\star\odot z^\star) < 1$. Second, for all $d\in\R^n$, we have $[z^\star\odot d]_i=0$ if $[z^\star]_i^2=0$. Then the first line of \eqref{H2:PSD} follows from \eqref{eq:KKT:C2:2}.
    
\paragraph{Show the second line of \eqref{H2:PSD}}  First, if $\|z^\star\|=1$, we have $1_n^\top (z^\star\odot z^\star)=1$. Also, $d\bot z^\star$ implies $ \langle 1_n,z^\star\odot d\rangle = \langle d, z^\star\rangle  =0$.  Second, it holds that $[z^\star\odot d]_i=0$ if $[z^\star]_i^2=0$. Then the second line of \eqref{H2:PSD} also follows from \eqref{eq:KKT:C2:2}.

   \item Suppose $z^\star$ is a second-order KKT point of Problem~\eqref{NC_2}. First,
    $z^\star$ satisfies \eqref{KKT:NC2}: there exists some $\lambda^\star\in\mathbb{R}$ such that
   \begin{align*}
 \nabla f(z^\star\odot z^\star)\odot z^\star &= \lambda_N^\star z^\star
 \\
\|z^\star\|^2 &\le1
\\
\lambda^\star &\le 0
\\
(\|z^\star\|^2-1))\lambda_N^\star &=0
   \end{align*}
 This implies that
   \begin{align*}
[\nabla f(z^\star\odot z^\star)]_k &= \lambda_N^\star,\ \text{for all }k: [z^\star]_k \neq 0
 \\
z^\star\odot z^\star &\ge 0
 \\
\|z^\star\|^2 &\le1
\\
\lambda_N^\star &\le 0
\\
(1_n^\top(z^\star\odot z^\star)-1))\lambda_N^\star &=0
   \end{align*}
 On the other hands, we  recognize that the following equations in \eqref{KKT:C2}:
   \begin{align*}
\nabla f(x^\star) &= \lambda^\star 1_n + \beta^\star
 \\
 \beta^\star &\ge 0
 \\
 x^\star\odot \beta^\star &=0
   \end{align*}
   are equivalent to 
   \begin{align*}
       \begin{cases}
         [\nabla f(x^\star)]_k  = \lambda^\star,\ \mathrm{for\ all\ } k: [x^\star]_k \neq0
         \\
          [\nabla f(x^\star)]_k \ge \lambda^\star,\ \mathrm{for\ all\ } k:  [x^\star]_k =0
       \end{cases}
   \end{align*}  
   Therefore, by choosing $\lambda^\star=\lambda_N^\star$, to show $z^\star\odot z^\star$
   is a first-order KKT point of \eqref{C_2}, it remains to show
      \begin{align}
          [\nabla f(z^\star\odot z^\star)]_k \ge \lambda_N^\star,\ \mathrm{for\ all\ } k:  [z^\star]_k =0
          \label{eqn:nabla f > lam}
   \end{align} 
   \paragraph{Case 1: $\{k: [z^\star]_k=0\}=\emptyset$} In this case, \eqref{eqn:nabla f > lam} holds automatically. Therefore, $z^\star\odot z^\star$ is  a first-order KKT point of \eqref{C_2}. 
   
    \paragraph{Case 2: $\{k: [z^\star]_k=0\}\neq\emptyset$}  
   Since $z^\star$ is a second-order KKT point of Problem~\eqref{NC_2}, 
       \begin{align}
    d^\top[\nabla^2 \mathcal{L}_N(z^\star,\lambda_N^\star)]d
     &= 4(z^\star\odot d)^\top \nabla^2 f(z^\star\odot z^\star)(z^\star\odot d)+2\langle\nabla f(z^\star\odot z^\star), d\odot d\rangle - 2\lambda_N^\star \|d\|^2
    \nonumber\\
     & \ge 0, \begin{cases}\text{for all }
d\in\R^n & \text{if }\|z^\star\|<1\\
\text{for all }d\bot z^\star & \text{if }\|z^\star\|=1
\end{cases}
    \label{H2:PSD:2}
    \end{align}
For any $k^\star \in\{k: [z^\star]_k =0\}$, construct $d=e_{k^\star}$, where $\{e_i\}$ denotes the standard basis in $\mathbb{R}^n$, and we have $z^\star\odot d=[z^\star]_k^2 e_k=0$ (and also $d\bot z^\star$).
Then, in either case of \eqref{H2:PSD:2}, such a construction $d=e_{k^\star}$ gives 
\[
          [\nabla f(z^\star\odot z^\star]_{k^\star} \ge \lambda_N^\star
\]
which then implies  \eqref{eqn:nabla f > lam}.
Therefore, $z^\star\odot z^\star$ is  a first-order KKT point of \eqref{C_2}.

Finally,  since $z^\star$ is a second-order KKT point of Problem~\eqref{NC_2} and 
\begin{align*}
[\nabla f(z^\star\odot z^\star)]_k &= \lambda_N^\star,\ \mathrm{for\ all\ } k:  [z^\star]_k \neq0
 \end{align*}   
we can further simplify \eqref{H2:PSD:2} as
\begin{align*}
    d^\top[\nabla^2 \mathcal{L}_N(z^\star,\lambda_N^\star)]d
     &= 4(z^\star\odot d)^\top \nabla^2 f(z^\star\odot z^\star)(z^\star\odot d)+2\langle\nabla f(z^\star\odot z^\star), d\odot d\rangle - 2\lambda_N^\star \|d\|^2
    \nonumber\\
     &= 4(z^\star\odot d)^\top \nabla^2 f(z^\star\odot z^\star)(z^\star\odot d)+2\sum_{k:[z^\star]_k=0} ([\nabla f(z^\star\odot z^\star)]_k-\lambda_N^\star) d_k^2
    \nonumber\\
     & \ge 0, 
     \begin{cases}\text{for all }
d\in\R^n & \text{if }\|z^\star\|<1\\
\text{for all }d\bot z^\star & \text{if }\|z^\star\|=1
\end{cases}
\end{align*}
By constructing 
\[
d_k=\begin{cases}
[u]_k:=0 &\text{if }[z^\star]_k=0
\\
\frac{[u]_k}{[z^\star]_k} &\text{if }[z^\star]_k\neq 0
\end{cases}
\]
and noting that $d\bot z^\star \iff 1^\top_n u=0$ and $\|z^\star\|=1 \iff 1_n^\top (z^\star\odot z^\star)=1$, we further have
\begin{align*}
 d^\top[\nabla^2 \mathcal{L}_N(z^\star,\lambda_N^\star)]d
 =4 u^\top \nabla^2 f(z^\star\odot z^\star)u
      \ge 0, \text{ for all $u$ such that}
     \begin{cases}
[u]_k=0 & \text{if }[z^\star]_k=0\\
1_n^\top u=0 & \text{if }1_n^\top (z^\star\odot z^\star)=1
\end{cases}
\end{align*}
Therefore, in view of \eqref{eq:KKT:C2:2}, $z^\star\odot z^\star$ is a second-order KKT point of \eqref{C_2}.

\end{enumerate}

\end{proof}

\subsection{The weighted probability simplex --- Proof of Theorem \ref{thm:H3}} 

Recall the original problem is
\begin{equation}
    x^{\star} \in \argmin_{x \in \Delta_{n}^a} f(x)
    \tag{$\mathrm{C_3}$}
\end{equation}
and the Hadamard parameterized problem is
\begin{equation}
    z^{\star} \in \argmin_{z \in \mathcal{B}_n^a}  f(z\odot z)
    \tag{$\mathrm{NC_3}$}
\end{equation}

\paragraph{First-order KKT conditions of Problem~\eqref{C_3}} 
The Lagrangian of Problem~\eqref{C_3} is 
\begin{align}
     \mathcal{L}_C(x,\lambda,\beta) = f(x) - \lambda (a^\top x-1) +\beta^\top x 
\end{align}
The first-order KKT conditions are given by there exist $\lambda^\star\in\mathbb{R}$ and $\beta^\star\in\mathbb{R}^n$ such that
\begin{subequations}
\begin{align}
\nabla f(x^\star) &= \lambda^\star a + \beta^\star \label{KKT:C3:1}
 \\
 x^\star& \ge 0
 \label{KKT:C3:2}\\
 \beta^\star &\ge 0 \label{KKT:C3:3}
 \\
 x^\star\odot \beta^\star &=0 \label{KKT:C3:4}
 \\
 a^\top x^\star &=1 \label{KKT:C3:5}
\end{align}
\label{KKT:C3}
\end{subequations}

\paragraph{Second-order KKT conditions of Problem~\eqref{C_3}} 
As 
\[
\nabla^2 f(x) =\nabla^2_x \mathcal{L}_C(x,\lambda,\beta)
\]
the second-order KKT conditions of Problem \eqref{C_3} are given by \eqref{KKT:C3} and 
\begin{equation}
    u^\top\nabla^2f(x^\star)u \geq 0 
    \label{eq:KKT:C3:2}
\end{equation}
for all $u$ satisfying 
\[
\begin{cases}
a^\top u = 0 & 
\\
[u]_k = 0  &\text{if }[x^\star]_k = 0
\end{cases}
\]

\paragraph{First-order KKT conditions of Problem~\eqref{NC_3}}
The Lagrangian of Problem~\eqref{NC_3} is
\begin{align}
    \mathcal{L}_N(z,\lambda_N) = f(z\odot z) - \lambda_N (\|z\|_{\mathrm{diag}(a)}^2-1)
\end{align}
Then the first-order KKT conditions of Problem~\eqref{NC_3} are given by that there exist $\lambda^\star\in\mathbb{R}$ such that 
\begin{subequations}
\begin{align}
\nabla f(z^\star\odot z^\star)\odot z^\star &= \lambda_N^\star a\odot z^\star \label{KKT:NC3:1}
 \\
\|z^\star\|^2_{\mathrm{diag}(a)} &=1 \label{KKT:NC3:2}
\end{align}
\label{KKT:NC3}
\end{subequations}

\paragraph{Second-order KKT conditions of Problem~\eqref{NC_3}}
The second-order KKT conditions of Problem~\eqref{NC_3} are given by that there exist $\lambda_N^\star\in\mathbb{R}$ such that $(z^\star,\lambda_N^\star)$ satisfy \eqref{KKT:NC3} and 
\begin{align}
d^{\top}\left[\nabla^2 \mathcal{L}_N(z^\star,\lambda_N^\star)\right]d&=
d^{\top}\left[2\mathrm{diag}(\nabla f(z^\star\odot z^\star)) + 4 \mathrm{diag}(z^\star) \nabla^2 f(z^\star\odot z^\star) \mathrm{diag}(z^\star) - 2\lambda_N^\star \mathrm{diag}(a) \right]d
\nonumber\\
&= 4(z^\star\odot d)^\top \nabla^2 f(z^\star\odot z^\star)(z^\star\odot d) + 2\langle\nabla f(z^\star\odot z^\star), d\odot d\rangle  - 2\lambda_N^\star \|d\|_{\mathrm{diag}(a)}^2 \ge 0
 \label{KKT2:NC3}
\end{align}
for all $d$ such that
\[
d^\top \diag(a) z^\star =0
\]

\begin{proof}[\bf Proof of Theorem \ref{thm:H3}]
~

\begin{enumerate} 
 \item Suppose that $x^\star$ is a second-order KKT point of \eqref{C_3}. 
 Take any $z^\star$ from $\{z: z\odot z = x^\star\}$.  We first show $z^\star$ satisfies the first-order KKT conditions \eqref{KKT:NC3:1} and \eqref{KKT:NC3:2}. Multiplying both sides of the optimality condition \eqref{KKT:C3:1} by $z^\star$:
    \begin{align}
        & \nabla f(x^{\star})\odot z^{\star} = \lambda^{\star}a\odot z^{\star} + \beta^{\star}\odot z^{\star} \\
    \Longrightarrow &  f(z^{\star}\odot z^{\star})\odot z^{\star} = \lambda^{\star} a\odot z^{\star} + \beta^{\star}\odot z^{\star} \quad   \label{eq:Reduce_to_KKT1_NVX:W}
    \end{align}
    By complementary slackness \eqref{KKT:C3:4}: $z^\star\odot z^\star\odot\beta^\star=0$, and so $z^\star\odot\beta^\star=0$ by \ref{H2}, thus \eqref{eq:Reduce_to_KKT1_NVX:W} reduces to \eqref{KKT:NC3:1} by choosing $\lambda_N^\star=\lambda^\star$. Also note \eqref{KKT:C3:5} is equivalent to \eqref{KKT:NC3:2}: 
    \begin{equation*}
        1 = a^{\top}x^{\star} = a^{\top}z^{\star}\odot z^{\star} =  \|z^{\star}\|_{\mathrm{diag}(a)}^{2}
    \end{equation*}
    It remains to show $z^\star\in \mathcal{Z}^\star$ satisfies the second-order KKT conditions \eqref{KKT2:NC3} of Problem~\eqref{NC_3} with $\lambda_N^\star=\lambda^\star$.
    Since $x^\star$ satisfies \eqref{KKT:C3:1} and \eqref{KKT:C3:3}:
    \begin{align*}
       & \nabla f(x^\star) = \lambda^\star a + \beta^\star ,\beta^{\star}\ge 0
       \\
       \iff & \mathrm{diag}(\nabla f(z^\star\odot z^\star)) - \lambda^\star \mathrm{diag}(a) = \mathrm{diag}(\beta^\star)\succeq 0
       \\
       \iff & d^{\top}\left[\mathrm{diag}(\nabla f(z^\star\odot z^\star)) - \lambda^\star \mathrm{diag}(a)\right]d \geq 0 \\
      \overset{\text{\ref{H3}}}{\iff} & \langle \nabla f(z^\star\odot z^\star), d\odot d\rangle - \lambda^\star \|d\|_{\mathrm{diag}(a)}^2 \ge 0,\ \forall d  
    \end{align*}
    Plugging this into \eqref{KKT2:NC3} by choosing $\lambda_N^\star=\lambda^\star$, we get 
    \begin{align}
     d^{\top}[\nabla^2 \mathcal{L}_N(z^\star)]d
     &= 4(z^\star\odot d)^\top \nabla^2 f(z^\star\odot z^\star)(z^\star\odot d)+2\langle\nabla f(z^\star\odot z^\star), d\odot d\rangle - 2\lambda^\star \|d\|_{\mathrm{diag}(a)}^2
      \nonumber\\
     &\ge 4(z^\star\odot d)^\top \nabla^2 f(z^\star\odot z^\star)(z^\star\odot d)
    \nonumber \\
     & \ge 0
    \label{H3:PSD}
    \end{align}
    for all $d$ such that
\[
d^\top \diag(a) z^\star =0
\]
     \eqref{H3:PSD} follows from \eqref{eq:KKT:C3:2} by noting that
    \[
    \begin{cases}
    a_n^\top (z^\star\odot d)= d^\top \diag(a) z^\star =0 &
    \\
    [z^\star\odot d]_k = [z^\star]_k [d]_k =0 &\text{if }  [z^\star]_k^2=0
    \end{cases}
    \]
    Therefore, $z^\star$ is a second-order KKT point of \eqref{NC_3}.
    
   \item Suppose  $z^\star$ is second-order KKT points of Problem~\eqref{NC_3}. 
  First, note that $z^{\star}$ satisfies \eqref{KKT:NC3:2}, we have
    \begin{equation*}
       z^\star\odot z^\star \ge 0 \text{ and } a^\top (z^\star\odot z^\star) = 1
    \end{equation*}
    That is, $z\odot z$ satisfies the KKT conditions \eqref{KKT:C3:2} and \eqref{KKT:C3:5} of Problem~\eqref{C_3}. Second, since $z^\star$ satisfies \eqref{KKT:NC3:1}, there exists $\lambda_N^\star\in\mathbb{R}$ such that
   \begin{align*}
       & \nabla f(z^\star\odot z^\star) \odot z^\star = \lambda_N^\star a \odot z^\star
       \\
       \Longrightarrow & [\nabla f(z^\star\odot z^\star)]_k = \lambda_N^\star [a]_k,\ \ \mathrm{for\ all\ } k: [z^\star]_k \neq 0
   \end{align*}
On the other hand, we recognize that \eqref{KKT:C3:1},\eqref{KKT:C3:3} and \eqref{KKT:C3:4} with $\lambda^\star=\lambda^\star_N$ are equivalent to:
   \begin{align*}
       \begin{cases}
         [\nabla f(x^\star)]_{k}  = \lambda^\star_N [a]_k ,\ \ \mathrm{for\ all\ } k: [x^\star]_k \neq0
         \\
          [\nabla f(x^\star)]_{k} \ge \lambda^\star_N [a]_k,\ \ \mathrm{for\ all\ } k:[x^\star]_k =0
       \end{cases}
   \end{align*}
Therefore, it remains to show
   \begin{align}
       [\nabla f(z^\star\odot z^\star)]_k &\ge \lambda_N^\star [a]_k ,\ \mathrm{for\ all\ } k: [z^\star]_k = 0
       \label{eqn:nabla > lam a}
   \end{align}
   Then $z^\star\odot z^\star$ satisfies the KKT conditions \eqref{KKT:C3:1},\eqref{KKT:C3:3} and \eqref{KKT:C3:4} with
   \begin{align*}
    \lambda^\star&=\lambda_N^\star
    \\
    [\beta^{\star}]_k &= [\nabla f(z^\star\odot z^\star)]_k - \lambda^\star [a]_k  \ge 0,\ \forall k
   \end{align*}
   Consequently, $z^{\star}\odot z^{\star}$ is a first-order KKT point of Problem~\eqref{C_3}.
   
 \paragraph{Show \eqref{eqn:nabla > lam a}} This is trivial if $\{k: [z^\star]_k=0\}=\emptyset$. Now suppose $\{k: [z^\star]_k=0\}\neq \emptyset$. Since $z^\star$ is a second-order KKT point of \eqref{NC_3},  \eqref{KKT2:NC3} claims that 
 \begin{align*}
d^{\top}\left[\nabla^2 \mathcal{L}_N(z^\star,\lambda_N^\star)\right]d&=4(z^\star\odot d)^\top \nabla^2 f(z^\star\odot z^\star)(z^\star\odot d)+
 2\langle\nabla f(z^\star\odot z^\star), d\odot d\rangle  - 2\lambda_N^\star \|d\|_{\mathrm{diag}(a)}^2 \ge 0
\end{align*}
for all $d$ such that $d^\top \diag(a) z^\star =0$. Now, for any $k^\star\in\{k: [z^\star]_k=0\}$, choosing
\[
d=e_{k^\star}
\]
which satisfies that  $d^\top \diag(a) z^\star = [a]_{k^\star} [z^\star]_{k^\star} =0$, we have
 \begin{align*}
d^{\top}\left[\nabla^2 \mathcal{L}_N(z^\star,\lambda_N^\star)\right]d&=
 2[\nabla f(z^\star\odot z^\star)]_{k^\star} - 2\lambda_N^\star [a]_{k^\star} \ge 0
\end{align*}
This completes the proof of \eqref{eqn:nabla > lam a}.

Finally, we show $z^\star\odot z^\star$ is a second-order KKT point of \eqref{C_3}. Again, since $z^\star$ is a second-order KKT point of Problem~\eqref{NC_3} and 
\begin{align*}
[\nabla f(z^\star\odot z^\star)]_k &= \lambda_N^\star [a]_k,\ \mathrm{for\ all\ } k:  [z^\star]_k \neq0
 \end{align*}   
we can further simplify \eqref{KKT2:NC3} as
 \begin{align*}
d^{\top}\left[\nabla^2 \mathcal{L}_N(z^\star,\lambda_N^\star)\right]d&=4(z^\star\odot d)^\top \nabla^2 f(z^\star\odot z^\star)(z^\star\odot d)+
 2\sum_{k: [z^\star]_k=0} (\nabla f(z^\star\odot z^\star)-\lambda_N^\star [a]_k) [d]_k^2
 \ge 0
\end{align*}
for all $d$ such that $d^\top \diag(a) z^\star =0$. By constructing 
\[
[d]_k =
\begin{cases}
[u]_k:=0 & \text{if }[z^\star]_k=0
\\
\frac{[u]_k}{[z^\star]_k} & \text{if }[z^\star]_k\neq0
\end{cases}
\Longrightarrow
d^\top \diag(a) z^\star = a^\top u =0
\]
we have 
 \begin{align*}
d^{\top}\left[\nabla^2 \mathcal{L}_N(z^\star,\lambda_N^\star)\right]d&=4 u^\top \nabla^2 f(z^\star\odot z^\star) u
 \ge 0 
\end{align*}
for all $u$ satisfying
\[
\begin{cases}
a^\top u = 0 &
\\
[u]_k = 0 & \text{if } [z^\star\odot z^\star]_k=0
\end{cases}
\]
This completes the proof of showing $z^\star\odot z^\star$ is a second-order KKT point of \eqref{C_3}.
   
\end{enumerate}  
\end{proof}

\subsection{The \texorpdfstring{$\ell_1$}{} norm ball --- Proof of Theorem \ref{thm:H4}} 
Recall the original problem is 
\begin{equation}
    x^{\star} \in \argmin_{x\in\mathcal{B}^{1}_n}f(x)
    \label{C_4}
    \tag{$\mathrm{C_4}$}
\end{equation}
and the Hadamard parameterized problem is
\begin{equation}
 (z_u^\star,z_v^\star) \in  \argmin_{(z_u,z_v)\in \mathcal{B}^{2}_{2n}}f(z_u\odot z_u - z_v\odot z_v)
    \tag{$\mathrm{NC_4}$}
    \label{NC_4}
\end{equation}

\paragraph{First-order KKT conditions of Problem~\eqref{C_4}} The Lagrangian of Problem~\eqref{C_4} is
\begin{equation}
    \mathcal{L}_C(x,\lambda) = f(x) - \lambda\left(\|x\|_1 - 1\right).
\end{equation}
Since Problem~\eqref{C_4} is convex,  the global optimality conditions are given by the first-order KKT conditions that there exists $\lambda^{\star} \in \bbR$ such that:
\begin{subequations}
\begin{align}
 \nabla f(x^{\star}) & \in \lambda^{\star}\sign(x^\star) \label{KKT:C4:1} \\
 \lambda^{\star} & \leq 0 \label{KKT:C4:2} \\
 \|x^{\star}\|_1 & \leq 1 \label{KKT:C4:3} \\
 \lambda^{\star}\left(\|x^{\star}\|_1 - 1\right) &= 0 \label{KKT:C4:4}
\end{align}
 \label{KKT:C4}
\end{subequations}
where 
\[
[\sign(x)]_k=\begin{cases}
 1 & \mathrm{if}~ [x]_k>0
 \\
 -1&\mathrm{if}~ [x]_k<0
 \\
 [-1,1] &\mathrm{if}~ [x]_k=0
\end{cases}\]  

\paragraph{First-order KKT conditions of Problem~\eqref{NC_4}} The Lagrangian of Problem~\eqref{NC_4} is
\begin{equation}
    \mathcal{L}_N(z_u,z_v,\lambda_N) = f(z_u\odot z_u - z_v\odot z_v) - \lambda_N\left(\|z_u\|_2^{2} + \|z_v\|_2^2 - 1\right)
\end{equation}
Thus we derive the first-order KKT conditions:
\begin{subequations}
\begin{align}
    \nabla f(\tilde{x})\odot z_u^{\star} &= \lambda_N^{\star}z_u^{\star} \label{KKT:NC4:1} \\
     -\nabla f(\tilde{x})\odot z_v^{\star} &= \lambda_N^{\star}z_v^{\star} \label{KKT:NC4:2} \\
     \|z_u^{\star}\|_2^{2} + \|z_v^{\star}\|_{2}^{2} & \leq 1 \label{KKT:NC4:3} \\
     \lambda_N^{\star} & \leq 0 \label{KKT:NC4:4} \\
     \lambda_N^{\star}\left(\|z_u^{\star}\|_2^{2} + \|z_v^{\star}\|_2^2 - 1\right) &= 0 \label{KKT:NC4:5}
\end{align}
\label{KKT:NC4}
\end{subequations}
where we denote $\tilde{x}:=z_u^\star\odot z_u^\star-z_v^\star\odot z_v^\star$ to simplify notations.

\paragraph{Second-order KKT conditions of Problem~\eqref{NC_4}} First of all, let us compute the Hessian of $\mathcal{L}_N$:
\begin{dmath}
   \nabla^2\mathcal{L}_N(z_u^\star,z_v^\star,\lambda_N^\star) = 2
  \begin{bmatrix}
  \diag(\nabla f(\tilde{x}))-\lambda_N^\star I_{n} &
  \\
  & -\diag(\nabla f(\tilde{x}))-\lambda_N^\star I_{n}
  \end{bmatrix} 
   + 4
\begin{bmatrix}
\diag(z_u^\star) 
\\
 -\diag(z_v^\star)
\end{bmatrix} 
 \nabla^2 f(\tilde{x})
\begin{bmatrix}
\diag(z_u^\star) 
\\
 -\diag(z_v^\star)
\end{bmatrix}^\top 
\label{eqn:Hessian:L1}
\end{dmath}
The second-order KKT conditions is given by 
\begin{equation}
2
  \begin{bmatrix}
  \diag(\nabla f(\tilde{x}))-\lambda_N^\star I_{n} &
  \\
  & -\diag(\nabla f(\tilde{x}))-\lambda_N^\star I_{n}
  \end{bmatrix} 
   + 4
\begin{bmatrix}
\diag(z_u^\star) 
\\
 -\diag(z_v^\star)
\end{bmatrix} 
 \nabla^2 f(\tilde{x})
\begin{bmatrix}
\diag(z_u^\star) 
\\
 -\diag(z_v^\star)
\end{bmatrix}^\top \succeq 0
\label{KKT2:NC4}
\end{equation}

\begin{proof}[\bf Proof of Theorem \ref{thm:H4}]
Denote 
\[\mathcal{Z}^\star=\left\{(z_u^\star,z_v^\star):  z_u^\star\odot z_u^\star-z_v^\star\odot z_v^\star = x^\star, \lambda^\star z_u^\star\odot z_v^\star=0 \right\}\]
where $x^\star$ is any first-order KKT point of \eqref{C_4} and $\lambda^\star$ is the Lagrangian dual variable in \eqref{KKT:C4:2}.
As usual, we break the proof two  parts.
\begin{enumerate} 
    \item  Suppose that $x^\star$ is a second-order KKT point of \eqref{C_4}. Take any $(z_u^\star,z_v^\star)\in \mathcal{Z}^\star$. We first show $(z_u^\star,z_v^\star) \in \sZ^{\star}$ satisfies the first-order KKT conditions \eqref{KKT:NC4:1}--\eqref{KKT:NC4:5} for Problem~\eqref{NC_4}.
    First of all, since $(z_u^\star,z_v^\star)\in\mathcal{Z}^\star$, we have $\tilde{x}=x^\star$. Then the first-order  KKT conditions \eqref{KKT:NC4:3}--\eqref{KKT:NC4:5} directly follow from  \eqref{KKT:C4:2}--\eqref{KKT:C4:4} by choosing $\lambda_N^\star=\lambda^\star$.
    It suffices to show the first-order KKT conditions \eqref{KKT:NC4:1}--\eqref{KKT:NC4:2} with $\lambda_N^\star=\lambda^\star$.
    Rewriting \eqref{KKT:C4:1} we obtain:
    \begin{equation}
    \begin{aligned}
        \nabla [f(x^{\star})]_i = \lambda^{\star} & \quad \text{ if } [x^{\star}]_i > 0 \\
        \nabla [f(x^{\star})]_i = -\lambda^{\star} & \quad \text{ if } [x^{\star}]_i < 0 \\
        \nabla [f(x^{\star})]_i \in [\lambda^{\star}-\lambda^{\star}] & \quad \text{ if } [x^{\star}]_i = 0 \\
    \end{aligned}
    \label{KKT:C4:NEW}
    \end{equation}

    \begin{description}
        \item[Case 1:] $\lambda^\star=0$.  Then we have $\nabla f(x^\star)=0$, {\em i.e.} $\nabla f(\tilde{x})=0$.  Take $\lambda_N^\star=\lambda^\star=0$. Then it is trivial that \eqref{KKT:NC4:1}--\eqref{KKT:NC4:2} is true.
        \item[Case 2:] $\lambda^\star\neq 0$. Then the above equation \eqref{KKT:C4:NEW}  implies the decoupling property 
        \[u^\star\odot v^\star=0.\]
        Therefore, 
\[
[x^\star]_i=[\tilde{x}]_i=
\begin{cases}
 {[z_u^\star]_i}^2 & \text{if } [z_u^\star]_i\neq 0
 \\
 -{[z_v^\star]_i}^2 & \text{if } [z_v^\star]_i\neq 0 
\end{cases}
\]
        which further indicates that for all $(z_u^\star,z_v^\star) \in \sZ^{\star}$:
    \begin{align*}
        \nabla [f(\tilde{x})]_i = \lambda^{\star} & \quad \text{ if } [\tilde{x}]_i=[x^{\star}]_i > 0 \iff [z_u^\star]_i \neq 0, [z_v^\star]_i=0 \\
        \nabla [f(\tilde{x})]_i = -\lambda^{\star} & \quad \text{ if } [\tilde{x}]_i=[x^{\star}]_i < 0 \iff [z_v^\star]_i \neq 0, [z_u^\star]_i=0 \\
        \nabla [f(\tilde{x})]_i \in [\lambda^{\star},-\lambda^{\star}] & \quad \text{ if } [\tilde{x}]_i=[x^{\star}]_i = 0  \iff [z_u^\star]_i=0, [z_v^\star]_i=0\\
    \end{align*}
Multiplying  $\nabla f(\tilde{x})$ by $z_u^\star$ and $z_v^\star$ respectively, and combining the above three, we then get  the KKT conditions \eqref{KKT:NC4:1} and \eqref{KKT:NC4:2}, respectively. 
    \end{description}
    
    It remains to show $(z_u^\star,z_v^\star) \in \sZ^{\star}$ satisfies the second-order optimality condition \eqref{KKT2:NC4} with $\lambda_N^\star=\lambda^\star$. In this case, we still have $\tilde{x}=x^\star.$
Since $\nabla f(x^\star) \in \lambda^\star \sign(x^\star)$ by \eqref{KKT:C4:1},  we must have $\pm\diag(\nabla f(x^\star))-\lambda^\star I_n \succeq 0$ (recall that $\lambda^\star\le 0$). Therefore, we have 
\[
\pm\diag(\nabla f(\tilde{x}))-\lambda^\star I_n \succeq 0.
\]
Using this,  to show the second-order optimality condition \eqref{KKT2:NC4}:
\[
2
  \begin{bmatrix}
  \diag(\nabla f(\tilde{x}))-\lambda_N^\star I_{n} &
  \\
  & -\diag(\nabla f(\tilde{x}))-\lambda_N^\star I_{n}
  \end{bmatrix} 
   + 4
\begin{bmatrix}
\diag(z_u^\star) 
\\
 -\diag(z_v^\star)
\end{bmatrix} 
 \nabla^2 f(\tilde{x})
\begin{bmatrix}
\diag(z_u^\star) 
\\
 -\diag(z_v^\star)
\end{bmatrix}^\top \succeq 0
\]
it suffices to show that 
\begin{align}
4
\begin{bmatrix}
\diag(z_u^\star) 
\\
 -\diag(z_v^\star)
\end{bmatrix} 
 \nabla^2 f(\tilde{x})
\begin{bmatrix}
\diag(z_u^\star) 
\\
 -\diag(z_v^\star)
\end{bmatrix}^\top \succeq 0
\label{H4:PSD}
\end{align}
which follows from the convexity assumption of $f$.

\item Suppose $(z_u^\star,z_v^\star)$ is second-order KKT points of Problem~\eqref{NC_4}. The target is to show $(z_u^\star,z_v^\star)\in \mathcal{Z}^\star$.  For convenience, we show its contrapositive: If $(z_u^\star,z_v^\star)\notin \mathcal{Z}^\star$, then $(z_u^\star,z_v^\star)$ is not a second-order stationary point of Problem~\eqref{NC_4}. First of all, the combination of \eqref{KKT:NC4:1}--\eqref{KKT:NC4:2} shows that 
\begin{align*}
  \nabla f(\tilde{x})\odot z_u^{\star}\odot z_v^\star & = \lambda_N^\star  z_u^\star\odot z_v^\star
  =- \lambda_N^\star  z_u^\star\odot z_v^\star
\end{align*}
implying that
\begin{equation}
\lambda_N^\star  z_u^\star\odot z_v^\star=0
\label{eqn:lam:uv:0}
\end{equation}

Therefore, together the assumption that $(z_u^\star,z_v^\star)\notin\mathcal{Z}^\star$, we must have 
\begin{equation}
\tilde{x}\neq x^\star
\label{eqn:sub:opt}    
\end{equation}
This is because, otherwise, we can choose $\lambda^\star=\lambda^\star_N$ to certify the global optimality of $\tilde{x})$ in Problem~\eqref{C_4}, {\em i.e.} $(z_u^\star,z_v^\star)\in\mathcal{Z}^\star$, which is a contradiction.

To proceed, we further note that the first-order KKT conditions \eqref{KKT:NC4:3}--\eqref{KKT:NC4:5} directly implies  \eqref{KKT:C4:2}--\eqref{KKT:C4:4} by choosing $\lambda^\star=\lambda_N^\star$. This means that the suboptimality condition \eqref{eqn:sub:opt} implies that 
\begin{equation}
\nabla f(\tilde{x}) \notin \lambda_N^\star\sign(\tilde{x})
\label{eqn:grad:not:sign}
\end{equation}

Also in view of \eqref{KKT:NC4:1} and \eqref{KKT:NC4:2}, we have
\begin{subequations}
\begin{align}
    \nabla f(\tilde{x})\odot z_u^{\star} &= \lambda_N^{\star}z_u^{\star}
    \Longrightarrow  [\nabla f(\tilde{x})]_k=\lambda_N^\star\quad\mathrm{if}~[z_u^\star]_k \neq 0
    \label{KKT:NC4:1:nonzero} \\
     -\nabla f(\tilde{x})\odot z_v^{\star} &= \lambda_N^{\star}z_v^{\star}
      \Longrightarrow  [\nabla f(\tilde{x})]_k=-\lambda_N^\star\quad\mathrm{if}~[z_v^\star]_k \neq 0
     \label{KKT:NC4:2:nonzero}   
\end{align}
\end{subequations}

In the following, we proceed case by case.
\begin{description}
\item[Case 1:]  $\lambda_N^\star=0$. 
Then
\[
[\nabla f(\tilde{x})]_k=0 \quad\text{if either $[z_u^\star]_k\neq0$ or $[z_v^\star]_k\neq 0$}
\]
Note that we cannot have $\nabla f(\tilde{x})=0$, since then  $\lambda^\star=0$ could certify the global optimality of $\tilde{x}$ in Problem~\eqref{C_4}, which contradicts with the fact $(z_u^\star,z_v^\star)\notin \mathcal{Z}^\star$. Therefore, there must be some $k^\star$ such that 
\begin{equation}
[\nabla f(\tilde{x})]_{k^\star}\neq0 \quad\text{for some $[z_u^\star]_{k^\star}=[z_v^\star]_{k^\star}=0$}.
\label{H4:Key:1}
\end{equation}
In this case, plugging in $\lambda_N^\star=0$, we can further simplify the Hessian
\[
   \nabla^2\mathcal{L}_N(z_u^\star,z_v^\star,0) = 2
  \begin{bmatrix}
  \diag(\nabla f(\tilde{x}))  &
  \\
  & -\diag(\nabla f(\tilde{x})) 
  \end{bmatrix} 
   + 4
\begin{bmatrix}
\diag(z_u^\star) 
\\
 -\diag(z_v^\star)
\end{bmatrix} 
 \nabla^2 f(\tilde{x})
\begin{bmatrix}
\diag(z_u^\star) 
\\
 -\diag(z_v^\star)
\end{bmatrix}^\top 
\]
Then we can choose the direction vector as
\[
(d_u,d_v)=
\begin{cases}
(e_k^\star,0)&\text{if }[\nabla f(\tilde{x})]_{k^\star}<0  
\\
(0,e_k^\star)&\text{if }[\nabla f(\tilde{x})]_{k^\star}>0  
\end{cases}
\]
so that we have
\[
\begin{bmatrix}
d_u^\top &  d_v^\top
\end{bmatrix} \nabla^2\mathcal{L}_N(z_u^\star,z_v^\star,0) 
\begin{bmatrix}
d_u \\  d_v
\end{bmatrix}
=-2\left|[\nabla f(\tilde{x})]_{k^\star}\right|\overset{\eqref{H4:Key:1}}{<}0
\]
That being said, $(z_u^\star,z_v^\star)$ is not a second-order stationary point of Problem~\eqref{NC_4}.

\item[Case 2:] $\lambda_N^\star\neq 0$. Then by \eqref{eqn:lam:uv:0}, we must have the decoupling property
\[
z_u^\star\odot z_v^\star=0.
\]
Therefore, 
\[
\tilde{x}_k=
\begin{cases}
 {[z_u^\star]_k}^2 & \text{if } [z_u^\star]_k\neq 0 \Rightarrow [z_v^\star]_k=0
 \\
 -{[z_v^\star]_k}^2 & \text{if } [z_v^\star]_k\neq 0 \Rightarrow [z_u^\star]_k=0
\end{cases}
\]
We claim that there must be some $k^\star$ such that 
\[
[z_u^\star]_{k^\star}=[z_v^\star]_{k^\star}=0
\]
because otherwise $\tilde{x}$ has no zero entries, and the equations \eqref{KKT:NC4:1:nonzero}--\eqref{KKT:NC4:2:nonzero} then implies that
\[
\nabla f(\tilde{x}) = \lambda_N^\star \sign(\tilde{x})  
\]
which contradicts with the fact \eqref{eqn:grad:not:sign}:
\[
\nabla f(\tilde{x}) \notin \lambda_N^\star\sign(\tilde{x})
\]

Thus, we not only have $[z_u^\star]_{k^\star}=[z_v^\star]_{k^\star}=0$ for some $k^\star$, but also
\begin{equation}
\left|[\nabla f(\tilde{x})]_{k^\star}\right|>-\lambda_N^\star>0,
\label{H4:Key:2}
\end{equation}
since otherwise we still have $\nabla f(\tilde{x}) \in \lambda_N^\star\sign(\tilde{x})$.
Then we can choose the direction vector as
\[
(d_u,d_v)=
\begin{cases}
(e_k^\star,0)&\text{if }[\nabla f(\tilde{x})]_{k^\star}<0  
\\
(0,e_k^\star)&\text{if }[\nabla f(\tilde{x})]_{k^\star}>0  
\end{cases}
\]
so that we have
\[
\begin{bmatrix}
d_u^\top &  d_v^\top
\end{bmatrix} \nabla^2\mathcal{L}_N(z_u^\star,z_v^\star,\lambda_N^\star) 
\begin{bmatrix}
d_u \\  d_v
\end{bmatrix}
=-2\left|[\nabla f(\tilde{x})]_{k^\star}\right|-2\lambda_N^\star\overset{\eqref{H4:Key:2}}{<}0
\]
That being said, $(z_u^\star,z_v^\star)$ is not a second-order stationary point of Problem~\eqref{NC_4}.
\end{description}
\end{enumerate}
We now complete the proof of Theorem \ref{thm:H4}.  

\end{proof}

\section{Detailed Experimental Settings}
\label{sec:ExperimentalSettings}

\subsection{Computational Resources}
Our large-scale experiments (results presented in Figures~\ref{fig:LeastSquares_Interior}, \ref{fig:LeastSquares_Boundary} and \ref{fig:LeastSquares_Interior:2}) were run on a workstation with an Intel Core i9-9940X CPU and 128GB of RAM. Our small-scale experiments (results presented in Figures~\ref{fig:PGD_SpeedTest} and \ref{fig:SmallScaleLS}) were run on a 2012 MacBook Pro with an Intel Core i5 CPU and 16 GB of RAM. We estimate that approximately 150 hours of CPU time were used to run all our experiments as well as for prototyping. We did not use any GPU resources.

\subsection{Implementation of Algorithms}
We implemented our proposed algorithms, namely HadRGD, HadRGD-AW and HadRGD-BB, in Python. We used the implementation of Pairwise Frank-Wolfe included in the {\tt copt} package \cite{copt}. We implemented Projected Gradient Descent ourselves and used the algorithm described in \citep{duchi2008efficient} to compute projections to $\Delta_n$ (see Appendix~\ref{sec:SimplexProj}). We tested two line search schemes: Armijo rule along the feasible direction and Armijo rule along the projection arc \cite[Chpt. 3]{bertsekas1997nonlinear}. We found the performance of both schemes to be very similar, so in our experiments we used the feasible direction scheme. We implemented Entropic Mirror Descent (EMDA) exactly as described in Section 5 of \citep{beck2003mirror}, except we used a constant step-size rule instead of decaying step-sizes. 

\subsection{Hyperparameters}
\begin{table}[t]
\def\ROWCOLOR{black!10!white}
\begin{center}
    \begin{tabular}{l| c c  c}
    \toprule
       Algorithm & Hyperparameters & Case (i) & Case (ii) \\
    \midrule
        \multirow{3}{*}{PGD with linesearch} & \cellcolor{\ROWCOLOR} $s$ & \cellcolor{\ROWCOLOR}  20/L & \cellcolor{\ROWCOLOR} 20/L \\
           & $\beta$   & $0.75$  & $0.75$  \\
           & \cellcolor{\ROWCOLOR}$\rho_1$   &\cellcolor{\ROWCOLOR} $10^{-4}$  &\cellcolor{\ROWCOLOR}  $10^{-4}$ \\
        \midrule
        \multirow{4}{*}{HadRGD-AW} & $\alpha_{\mathrm{def}}$ & $10\sqrt{\frac{20n}{L}}$   & $10\sqrt{\frac{2n}{L}}$ \\
           &\cellcolor{\ROWCOLOR} $\beta$   &\cellcolor{\ROWCOLOR} $0.75$  &\cellcolor{\ROWCOLOR} $0.75$   \\
           & $\rho_1$   &  $10^{-4}$ & $10^{-4}$  \\
           &\cellcolor{\ROWCOLOR} $\rho_2$   &\cellcolor{\ROWCOLOR} $0.9$  &\cellcolor{\ROWCOLOR} $0.9$   \\
         \midrule
        \multirow{5}{*}{HadRGD-BB} & $\alpha_{\mathrm{def}}$ & $3.0$ &  $10\sqrt{\frac{2n}{L}}$ \\
           & \cellcolor{\ROWCOLOR}$\beta$  & \cellcolor{\ROWCOLOR} $0.5$ &\cellcolor{\ROWCOLOR} $0.75$  \\
           & $\eta$    &  $0.5$ & $0.5$  \\
           & \cellcolor{\ROWCOLOR}$\rho_1$   &\cellcolor{\ROWCOLOR} $0.1$  &\cellcolor{\ROWCOLOR} $0.1$  \\
    \bottomrule
    \end{tabular}
    \caption{Hyperparameters for benchmarking described in Section~\ref{sec:Experiments}. Case (i) is where $x_{\mathrm{true}}\in\mathrm{int}(\Delta_n)$ (see Figure~\ref{fig:LeastSquares_Interior}). Case (ii) is where $x_{\mathrm{true}}$ is on the boundary of $\Delta_n$ (see Figure~\ref{fig:LeastSquares_Boundary}). $L$ denotes the Lipschitz constant of the objective function $f(x)$.}
    \label{table: Hyperparameters}
\end{center}
\vspace{-0.1in}
\end{table}
For EMDA the only hyperparameter is the step size. As mentioned above, we chose to use a constant step size as empirically we observed this led to faster convergence than a decaying step size rule. For the implementation of PFW we used there are no free hyperparameters. For the experiment where $x_{\mathrm{true}}$ is on the boundary of $\Delta_n$ (see Figure~\ref{fig:LeastSquares_Interior}) we gave PFW the exact Lipschitz constant of $f(x)$, in an attempt to speed it up (we did not do so when $x_{\mathrm{true}} \in \mathrm{int}(\Delta_n)$. \\

As PGD with linesearch and HadRGD-AW employ a backtracking line search, there are several important parameters to set. Arguably the most important is the default step size ($\alpha_{\mathrm{def}}$ in Algorithms \ref{alg:HadRGD-AW}, \ref{alg:HadRGD-BB} and \ref{alg:PGDL}). We chose $\alpha_{\mathrm{def}}$ such that the line search sub-routine within PGD and HadRGD-AW took approximately  ten iterations to find a step size $\alpha_k$ satisfying the Armijo condition (Armijo-Wolfe conditions for HadRGD-AW). The other hyperparameters were set as described in Table~\ref{table: Hyperparameters}. For HadRGD-BB we experimented lightly with hyperparameter tuning, but found that it had little effect on the convergence speed. Thus, we stuck with the somewhat arbitrary values presented in Table~\ref{table: Hyperparameters}.

\section{Additional Experiments}

\label{sec:AdditionalExperiments}
\subsection{Projections to the simplex}
\label{sec:SimplexProj}
As discussed in Section~\ref{sec:Intro}, the convergence speed of PGD on $\Delta_n$ is strongly influenced by the computational complexity of the algorithm used to compute $P_{\Delta_n}$. For completeness, we test four different projection algorithms. Specifically, we test namely the well-known ``sort-then-project'' algorithm (see Figure 1 in \citep{duchi2008efficient}, henceforth referred to as {\tt SortProject})as well as those proposed in \citep{shalev2006efficient}, \citep{duchi2008efficient} and \citep{condat2016fast}, henceforth: {\tt PivotProject}, {\tt DuchiProject} and {\tt CondatProject} respectively. All four algorithms are implemented purely in Python. {\tt DuchiProject} and {\tt CondatProject} are exact implementations of the pseudocode presented in \citep{duchi2008efficient} and \citep{condat2016fast} respectively; {\tt SortProject} is based on the Matlab code provided in \citep{chen2011projection} while {\tt PivotProject} is the algorithm {\tt projection\_simplex\_pivot} available at \url{https://gist.github.com/mblondel/6f3b7aaad90606b98f71}. We tested Projected Gradient Descent (PGD) for \eqref{C_1} using all four algorithms for the underdetermined least squares problem, as described in Section~\ref{sec:Experiments}, for true solution $x_{\mathrm{true}}$ satisfying case (i) and case (ii). We note that although \citep{condat2016fast} provides convincing evidence that {\tt CondatProject} is faster than {\tt DuchiProject} when optimized and written in C, our experimental results suggest that {\tt DuchiProject} is the fastest projection algorithm  for $x_{\mathrm{true}}\in \mathrm{int}(\Delta)$ while {\tt SortProject} is the fastest for $x_{\mathrm{true}}$ on the boundary, when implemented purely in Python. As $x_{\mathrm{true}}\in \mathrm{int}(\Delta)$ is the case of primary interest for us, we use PGD with {\tt DuchiProject} in all our benchmarking experiments.

\begin{figure}[!ht]
    \centering
    \includegraphics[width=0.48\textwidth]{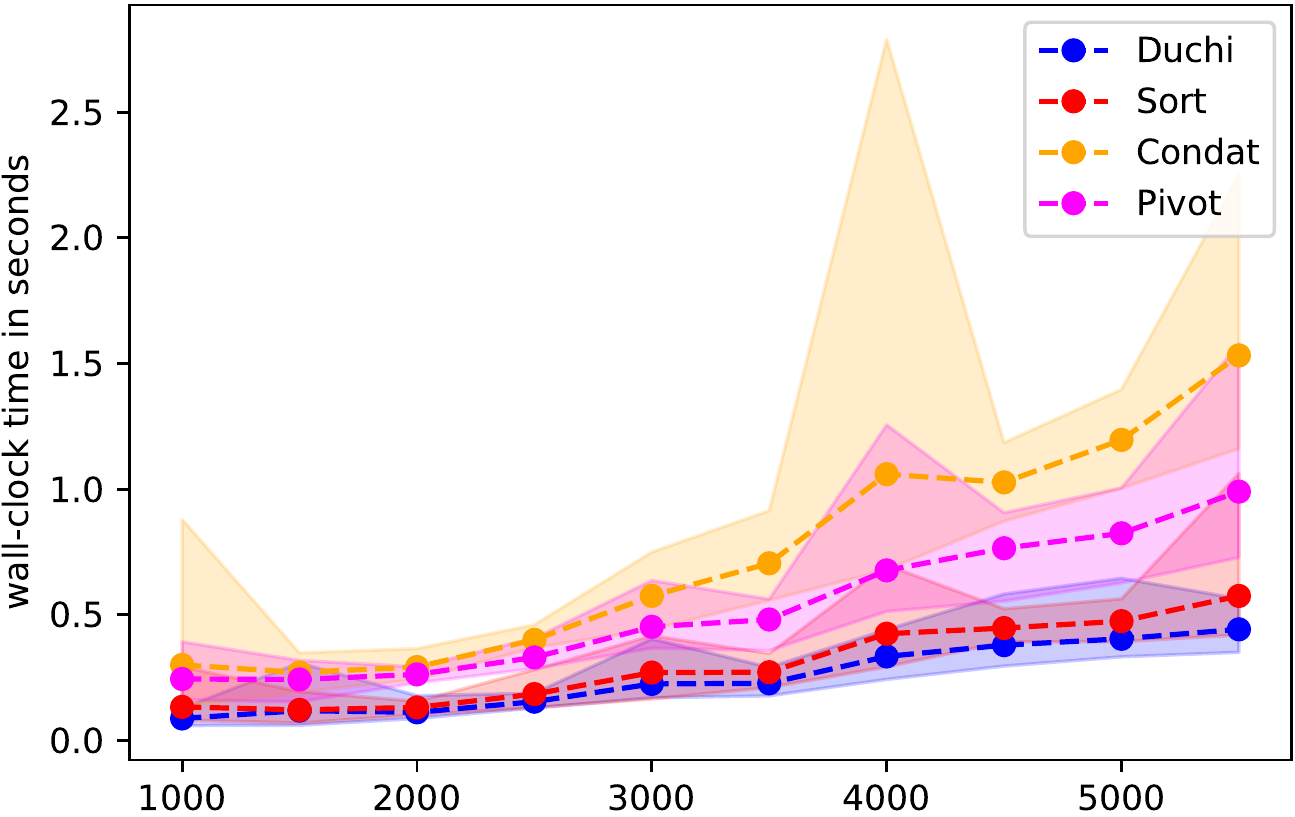}
    ~~~
    \includegraphics[width=0.48\textwidth]{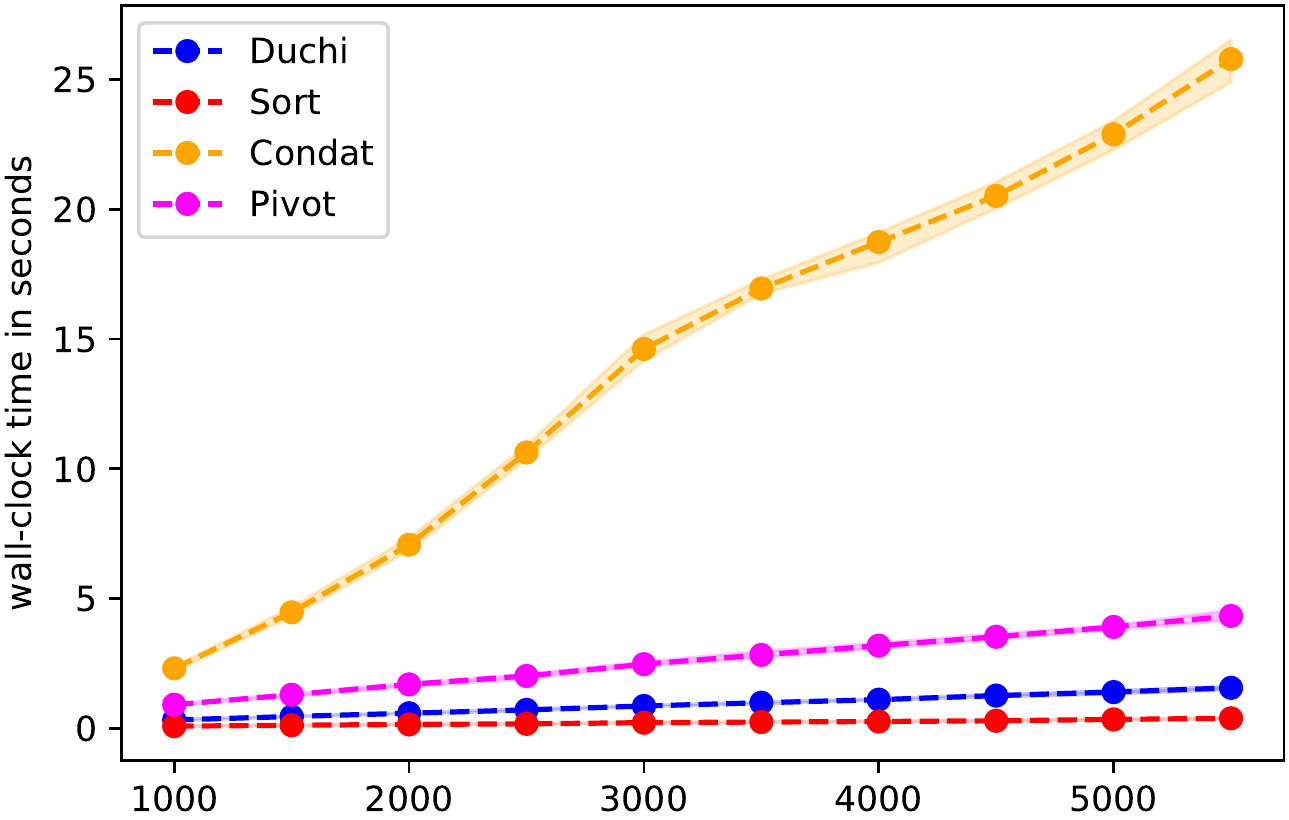}
    \caption{Time required to reach an error less than  for PGD with four different projection algorithms. {\bf Left:} $x_{\mathrm{true}} \in \mathrm{int}(\Delta_n)$. We use a target solution accuracy of $10^{-16}$. {\bf Right:} $x_{\mathrm{true}}$ a projection of a random Gaussian vector to $\Delta_n$. We use a target solution accuracy of $10^{-8}$ and a maximum of $200$ iterations. All results are averaged over ten trials and the shading denotes the min-max range.}
    \label{fig:PGD_SpeedTest}
\end{figure}

\subsection{Additional Benchmarking}
\label{sec:AdditionalBenchmarking}
Figure~\ref{fig:SmallScaleLS} contains additional benchmarking results where we compare HadRGD and PGD (both without line search) to EMDA and PFW using the so-called Demyanov-Rubinov step-size rule. As EMDA is significantly slower than the other three algorithms, we do not include it in our large-scale benchmarking. Finally, Figure~\ref{fig:AW_Linear} shows HadRGD-AW achieves a linear rate of convergence.
\begin{figure}[!ht]
    \centering
    \includegraphics[width=0.48\textwidth]{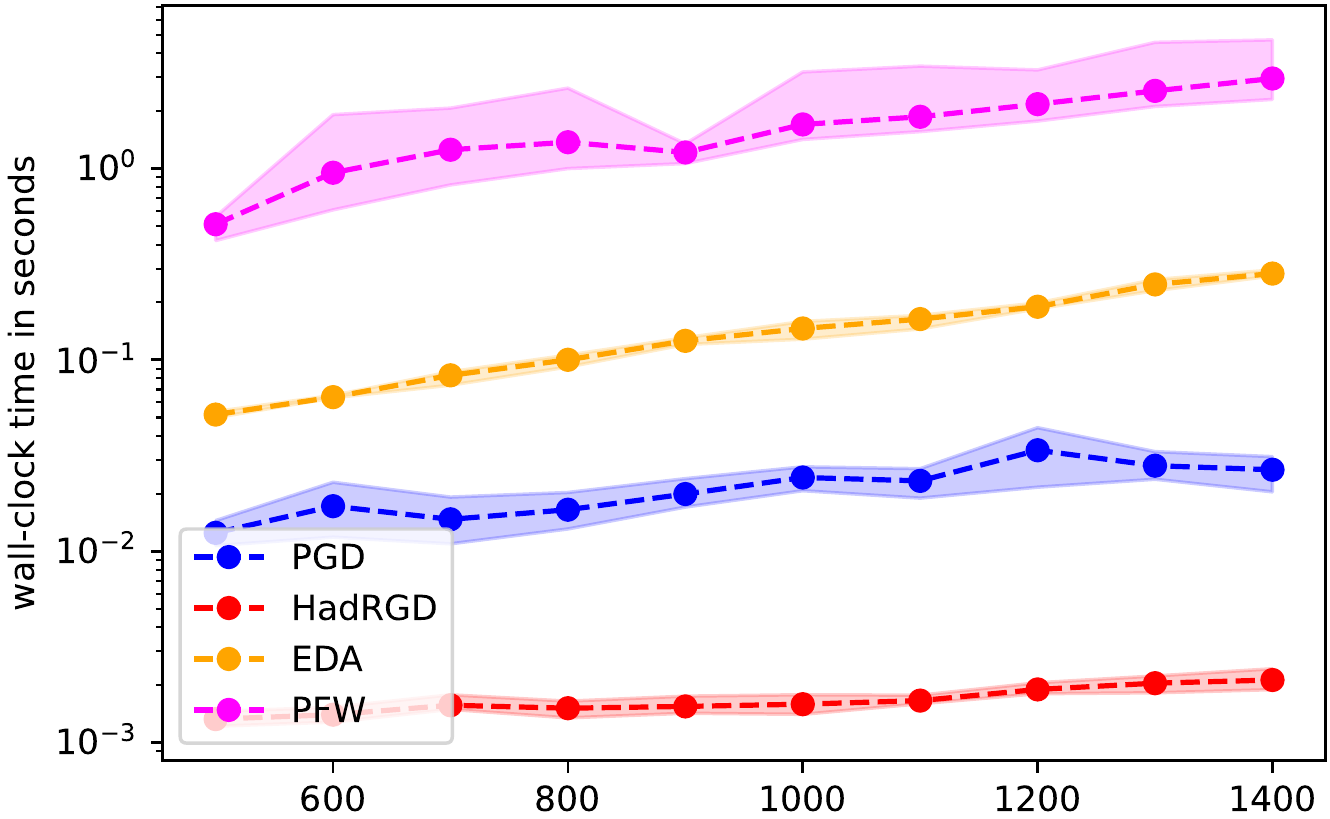}
    ~~~
    \includegraphics[width=0.48\textwidth]{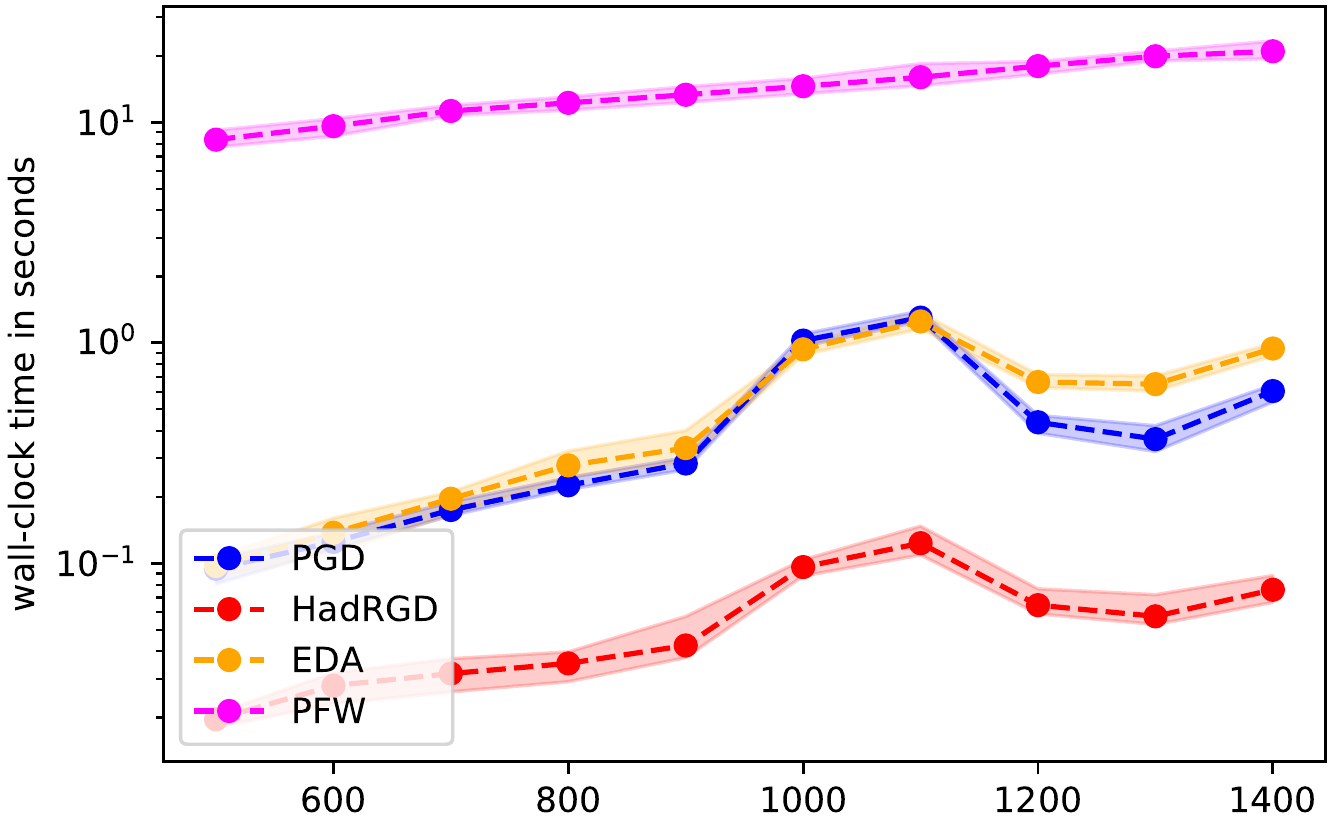}
    \caption{Time required to reach an error less than $10^{-8}$, or until the maximum number of iterations is hit. {\bf Left:} $x_{\mathrm{true}} \in \mathrm{int}(\Delta_n)$. {\bf Right:} $x_{\mathrm{true}}$ a projection of a random Gaussian vector to $\Delta_n$. All results are averaged over ten trials and the shading denotes the min-max range.}
    \label{fig:SmallScaleLS}
\end{figure}

\begin{figure}[!ht]
    \centering
    \includegraphics[width=0.47\linewidth]{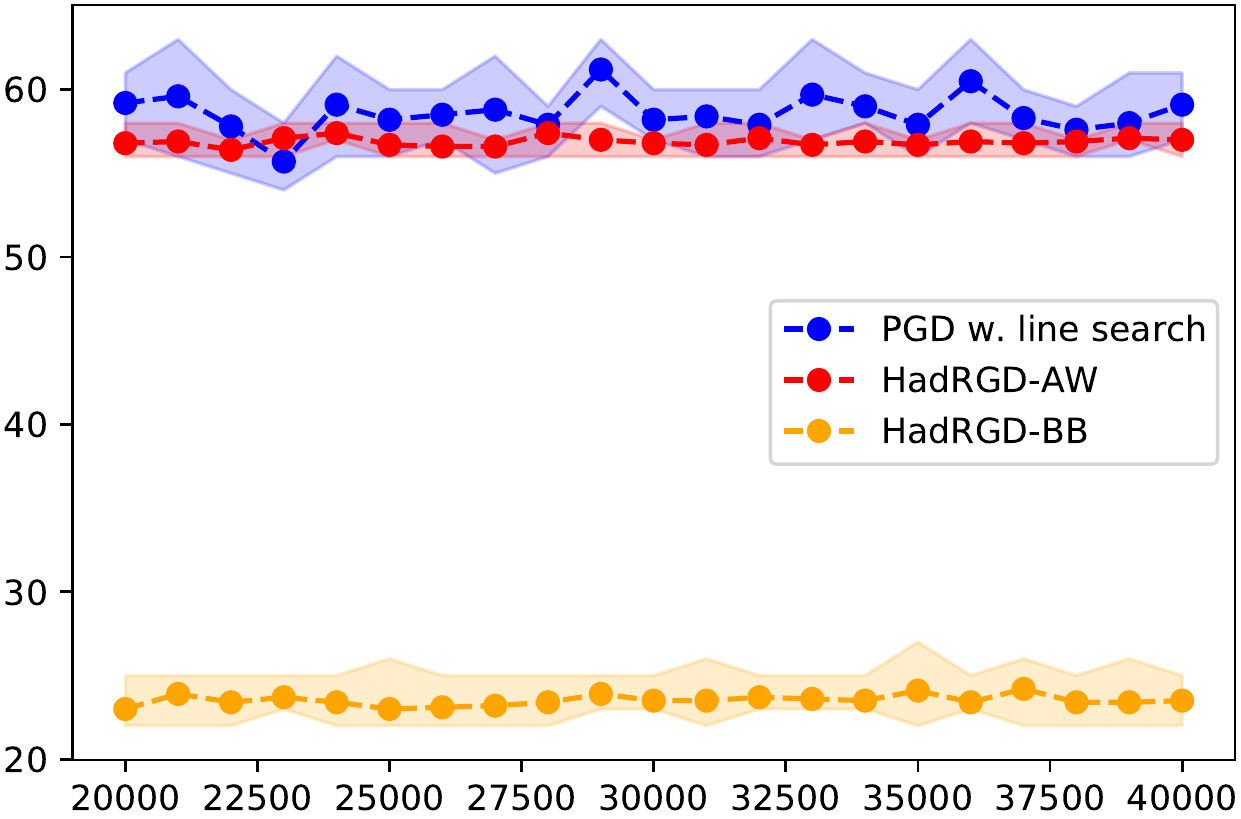}
    ~~~
    \includegraphics[width=0.47\linewidth]{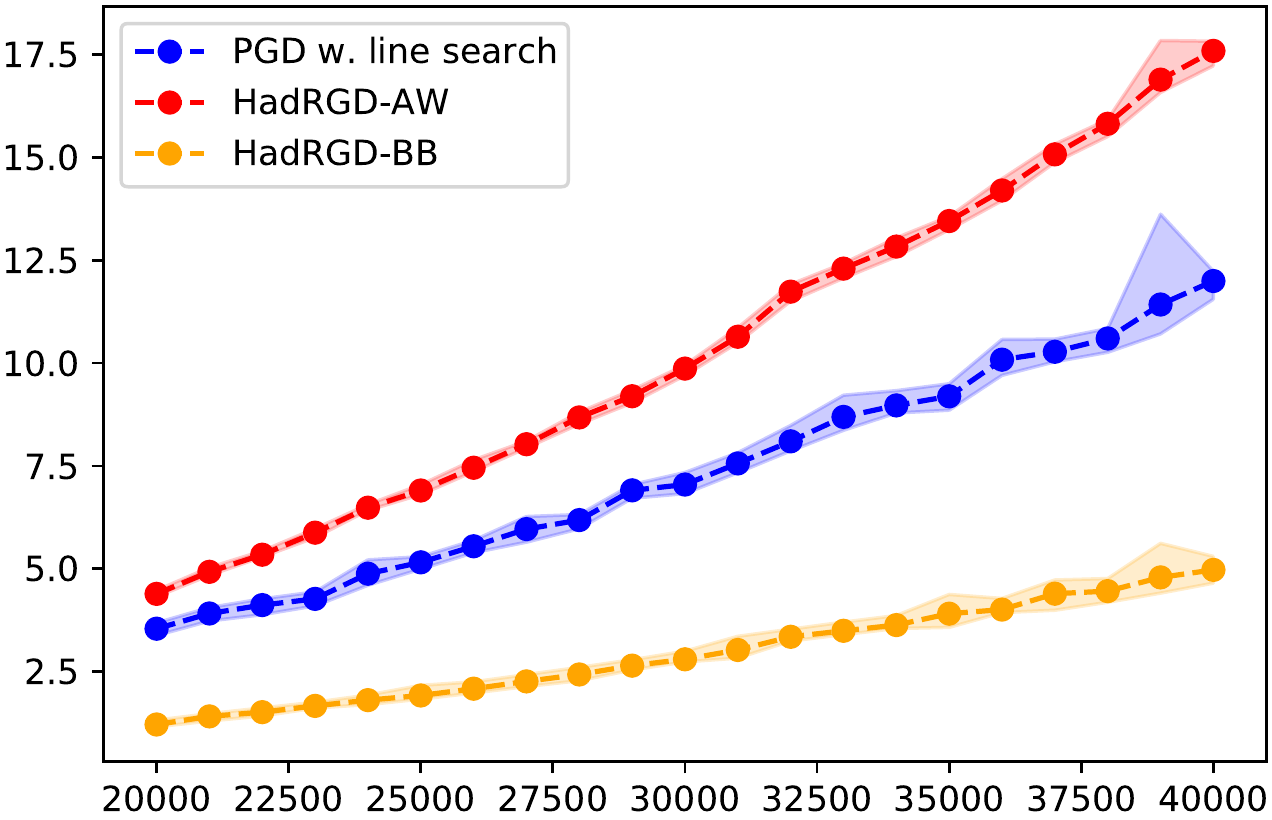}
    \caption{Solving the underdetermined least squares problem as in Section~\ref{sec:Experiments} with $x_{\mathrm{true}} \in \mathrm{int}(\Delta)$. Here, the target solution accuracy is $10^{-16}$. {\bf Left:} Number of iterations vs. $n$. {\bf Right:} Wall-clock time vs. $n$.}
    \label{fig:LeastSquares_Interior:2}
\end{figure}

\begin{figure}[!b]
    \centering
    \includegraphics[width=0.75\linewidth]{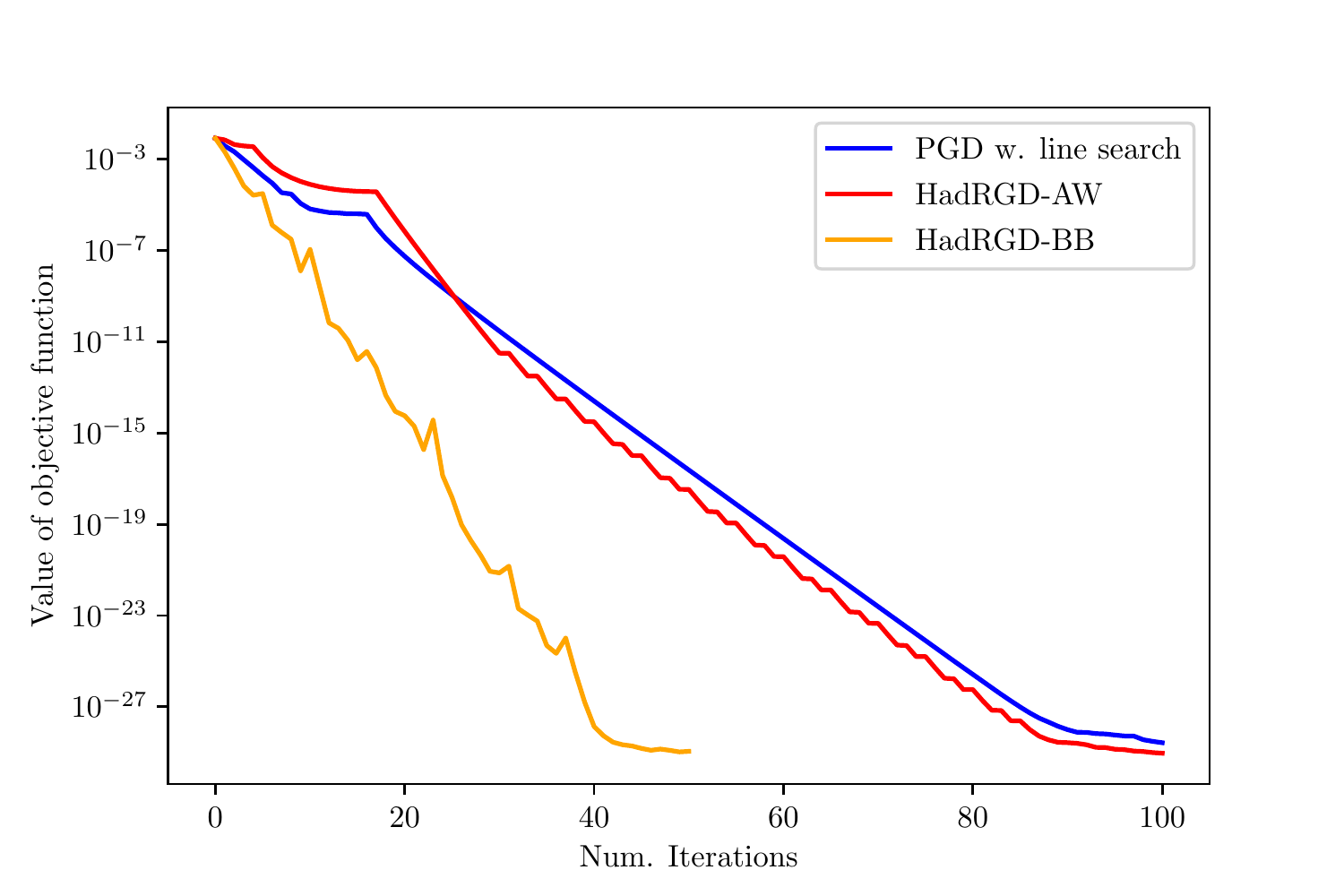}
    \caption{Solving the underdetermined least squares problem as in Section~\ref{sec:Experiments} with $x_{\mathrm{true}} \in \mathrm{int}(\Delta)$. HadRGD-AW, as well as HadRGD-BB and PGD with line search, enjoy a linear rate of convergence.}
    \label{fig:AW_Linear}
\end{figure}

\end{document}